\documentclass[3p,times]{elsarticle}

\usepackage{fontenc}
\usepackage{amssymb}
\usepackage{lipsum}
\usepackage{amsfonts}
\usepackage{graphicx}
\usepackage{epstopdf}
\usepackage{algorithmic}
\usepackage{latexsym}
\usepackage{color}
\usepackage{caption}
\usepackage{multirow}
\usepackage{setspace}
\usepackage{appendix}
\usepackage{float}
\usepackage{dsfont}
\usepackage{amsmath}
\usepackage{amscd}
\usepackage{mathrsfs}
\usepackage{amssymb}
\usepackage{multirow}
\usepackage{threeparttable}
\usepackage{booktabs}
\usepackage{rotating}
\usepackage{subeqnarray}
\usepackage{cases}
\usepackage{bbm}
\usepackage{amsthm}
\usepackage{txfonts}
\usepackage{bm}

\allowdisplaybreaks[4]

\def\cC{{\cal C}}
\def\cD{{\cal D}}

\def\cL{{\cal L}}
\def\cM{{\cal M}}
\def\cN{{\cal N}}

\def\cR{{\cal R}}

\def\cU{{\cal U}}

\def\be{\begin{eqnarray}}
\def\ee{\end{eqnarray}}
\def\ba{\begin{array}}
\def\ea{\end{array}}

\def\x{\times}
\def\R{\mathbb{R}}

\begin{document}

\newtheorem{lemma}{Lemma}[section]
\newtheorem{remark}{Remark}[section]
\newtheorem{example}{Example}[section]
\newtheorem{theorem}{Theorem}[section]
\newtheorem{corollary}{Corollary}[section]
\newtheorem{conjecture}{Conjecture}[section]
\newtheorem{definition}{Definition}[section]
\newtheorem{proposition}{Proposition}[section]
\newtheorem{condition}{Condition}[section]
\newtheorem{assumption}{Assumption}[section]

\numberwithin{equation}{section}

\begin{frontmatter}

\title{Discrete-time Approximation of Stochastic Optimal Control with Partial Observation}

\author{Yunzhang Li}
\ead{li\_yunzhang@fudan.edu.cn}
\address{Research Institute of Intelligent Complex Systems, Fudan University.}

\author{Xiaolu Tan}
\ead{xiaolu.tan@cuhk.edu.hk}
\address{Department of Mathematics, The Chinese University of Hong Kong.}

\author{Shanjian Tang}
\ead{sjtang@fudan.edu.cn}
\address{Department of Finance and Control Sciences, School of Mathematical Sciences, Fudan University.}

\begin{abstract}
We consider a class of stochastic optimal control problems with partial observation, and study their approximation by discrete-time control problems. We establish a convergence result by using weak convergence technique of Kushner and Dupuis [\textit{Numerical Methods for Stochastic Control Problems in Continuous Time} (2001), Springer-Verlag, New York], together with the notion of relaxed control rule introduced by El Karoui, H\.u\`u Nguyen and Jeanblanc-Picqu\'e [\textit{SIAM J. Control Optim.}, \textbf{26} (1988) 1025-1061]. In particular, with a well chosen discrete-time control system, we obtain a first implementable numerical algorithm (with convergence) for the partially observed control problem. Moreover, our discrete-time approximation result would open the door to study convergence of more general numerical approximation methods, such as machine learning based methods. Finally, we illustrate our convergence result by the numerical experiments on a partially observed control problem in a linear quadratic setting.
\end{abstract}


\end{frontmatter}

\section{Introduction}
\thispagestyle{empty}
The optimal control problem with partial observation has been introduced and studied since decades, and has broad applications in physics, engineering, economy and finance, etc. Let us refer in particular to Fleming \cite{Fleming}, Bensoussan \cite{MR1191160}, and Pardoux \cite{MR705935} among many other pioneer works on the subject.

\vspace{0.5em}

	In this paper, we study a class of partially observed control problems as follows.
	Let $(\Omega,\mathscr{F},\mathbb{P})$ be a complete probability space, equipped with the filtration $\{\mathscr{F}_t\}_{0\leq t \leq T}$, 
	and a standard $\mathbb{R}^m \times \mathbb{R}^k$-valued Brownian motion $(B,W)$,
	$A$ be a compact separable metric space with distance $d_A$.
	Given a $A$-valued progressively measurable control process $u = (u_s)_{s \in [0,T]}$,
	and with the initial condition $x_0 \in \mathbb{R}^d$, 
	and the coefficient functions $\left(b,\sigma\right): [0,T]\times \mathbb{R}^d \times A \longrightarrow \mathbb{R}^d \x \R^{(d\times m)}$ and $p: [0,T]\times \mathbb{R}^d \longrightarrow \mathbb{R}^{k}$, the signal process $X^u$ and the observation process $Y$ in the system are given by
\begin{subequations}\label{strict}
\begin{numcases}{}
\displaystyle \label{strict1} X_t^u = x_0 + \int_0^t b \left( s,X_s^u,u_s \right) ds + \int_0^t \sigma \left( s,X_s^u,u_s \right) dB_s, \\
\displaystyle \label{strict2} Y_t= \int_0^t p \left( s,X_s^u \right)ds + W_t.
\end{numcases}
\end{subequations}
	In the partially observed control problem, an admissible control process $u$ is required to be progressively measurable w.r.t. the observation filtration $\mathbb{F}^Y = (\mathscr{F}^Y_t)_{t \in [0,T]}$ generated by $Y$, i.e. $\mathscr{F}_t^Y:=\sigma(Y_s,\,s\leq t)$.
	Namely, the information available to the controller at time $t$ is the observations $\{Y_s,\,s\leq t\}$.
	Let us denote by $\cU$ the collection of all $A$-valued $\mathbb{F}^Y$-progressively measurable processes,
	we shall consider the following optimal control problem under partial observation:
	\begin{equation}\label{introVS}
		V =\sup_{u\in\cU} J(u), 
		~\mbox{with}~
		J(u):= \mathbb{E}^\mathbb{P} \left[ \int_0^T K \left(s,X_s^u,u_s\right)ds + G\left(X^u_T\right) \right].
	\end{equation}

	The formulation in~\eqref{strict} is somehow ill-posed. 
	In fact, the observation process $Y$ in \eqref{strict2} is defined with a given control process $u$,
	at the same time, the admissible control process $u$ is required to be adapted to the filtration $\mathbb{F}^Y$ generated by $Y$.
	A classical way to solve the problem is the so-called reference probability approach.
	Concretely, let us introduce the reference probability $\mathbb{Q}$ in spirit of Zakai transformation 
	by 
	\begin{equation}\label{Ltu}
		\frac{d \mathbb{Q}}{d \mathbb{P}} := \left( L^u_T \right)^{-1},
		~\mbox{with}~
		L_t^u :=\exp\left[ \int_0^t p(s,X_s^u) d Y_s - \frac{1}{2} \int_0^t \left| p(s,X_s^u) \right|^2 d  s \right] .
	\end{equation}
	Then by Girsanov theorem,
	under the new probability measure $\mathbb{Q}$, $(B,Y)$ is a standard $\mathbb{R}^{m+k}$-valued Brownian motion.
	An admissible control $u \in \cU$ is defined as a $A$-valued process, progressively measurable w.r.t. the filtration generated by the Brownian motion $Y$.
	More importantly, an equivalent formulation of the optimal control under partial observation \eqref{introVS} is
	\begin{equation}\label{Ju1}
		V =\sup_{u\in\cU} J(u), 
		~\mbox{with}~
		J(u)= \mathbb{E}^{\mathbb{Q}} \left[ L_T^u \left(\int_0^T K(s,X_s^u,u_s)\,ds + G(X^u_T)\right) \right].
	\end{equation}

	Based on the formulation \eqref{Ju1}, different approaches have been applied to study the partially observed control problem.
	A first important approach is the Pontryagin’s maximum principle, which provides a first order necessary condition of the optimal control as well as the optimally controlled process by a coupled forward-backward system, see e.g. Bensoussan \cite{MR1191160}, Haussmann \cite{Haussmann}, Li and Tang \cite{MR1363354}, and Tang \cite{MR1626880}, etc.
	 A second important approach is the dynamic programming method, which leads to a characterization of the value function by a Hamilton-Jacobi-Bellman (HJB) equation.
	 Since the signal process $X^u$ is not $\mathbb{F}^Y$-adapted, one needs to take the conditional distribution of $X^u_t$ knowing $Y$ as underlying process to deduce the dynamic programming.
	 In the early literature, one usually assumes the existence of the density function of the conditional distribution of $X^u_t$ knowing $Y$, which can then be described by a controlled stochastic PDE, and the corresponding value function is a (viscosity) solution to an infinitely dimensional HJB equation, see e.g. Lions \cite{MR1019600}, Gozzi and Swiech \cite{Gozzi2000}. 
	More recently, without the density assumption, Bandini, Cosso, Fuhrman and Pham \cite{MR3907014} established the dynamic programming principle and obtains a novel master type HJB equation. Further, in~\cite{MR3809474}, they use the randomization technique to obtain a dual BSDE characterization. 
	We mention in particular the compacification approach developed by El Karoui, Nguyen and Jeanblanc \cite{MR957652}, where a notion of relaxed control rules has been introduced for the partial observation problem,
	which will be recalled and essentially used in our paper.
	
	\vspace{0.5em}
	
	Very few studies are given on the numerical approximation method for partially observed control problems in the literature. In Bensoussan and Runggaldier \cite{MR913557}, the authors considered a problem where only the drift coefficient $b$ is controlled, and studied the time discretization as well as the space discretization, in order to construct  an $\epsilon$-optimal control. It is not clear that the discretized scheme can be easily implemented.
	In Archibald, Bao, Yong and Zhou \cite{archibald2020efficient}, the authors introduced a numerical algorithm for data driven feedback control problem (including the partially observed control problem).
	It is based on a direct computation on the G\^ateaux derivative of the cost function, together with a simple time discretization. However, it stays as heuristic as the error analysis due to time discretization has not been discussed in the paper.

	\vspace{0.5em}
	
	In this paper, we will study the approximation methods for a class of optimal control problems with partial observation based on the formulation in \eqref{Ju1}.
	We follow the main idea of Kushner and Dupuis \cite{MR1800098} to consider a sequence of discrete-time controlled systems.
	Under appropriate conditions, and by considering the martingale problem formulation in \cite{MR957652},
	we deduce that the discrete-time controls converge to a relaxed control rule.
	In contrast to \cite{MR1800098} which considers the finite difference approximation method with a locally consistent approximating controlled Markov chain for a standard optimal control problem,
	we consider a general class of controlled discrete-time systems in order to approximate our optimal control problem under partial observation.
	We investigate the problem to obtain appropriate conditions and prove a general convergence result.
	In particular, when the discrete-time system is chosen to be a locally consistent approximating controlled Markov chain,
	so that both time and space are discretized,
	it leads to an implementable numerical approximation method with dynamic programming principle on the discrete-time system.
	Our general convergence result will also lead to the convergence result of this numerical approximation algorithm.
	We illustrate this by some numerical experiments for a partially observed control problem in a linear-quadratic setting.
	For more general discrete-time system, it will lead to a high dimensional (but discrete-time) control problem, and a potential approximation method would be the machine learning based method as suggested in Han and E \cite{han2016deep}, as well as many further development, see e.g. Han, Jentzen and E \cite{han2018solving}.
	Notice that, in the literature of the machine learning based numerical methods for the optimal control problem (or HJB equations), the formulation of the problem is usually in continuous time, but the algorithm is in discrete-time, and the error analysis due to time discretization is generally omitted.
	Our convergence result would open a door for the convergence analysis of the corresponding methods.

	\vspace{0.5em}

	The rest of the paper is organized as follows. In Section~\ref{sectionconvergenceresult}, we first give the assumptions on the coefficients of the controlled system and the reward functions. 
	We then introduce the time discretization scheme for the partially observable control problem and provide our main convergence result. 
	In Section~\ref{sectionFDtest}, we present an implementable numerical scheme and perform some numerical experiment for a linear-quadratic partially observed control problem. 
	In Section~\ref{sectionformulation}, we revisit El Karoui, Nguyen and Jeanblanc \cite{MR957652},
	and recall the strong, weak, and relaxed formulations of the partially observable control problem, as well as the corresponding equivalence result.
	Finally, in Section~\ref{sectionproofconvergence}, we provide the proof of our main convergence theorem.

	\vspace{0.5em}

Throughout the paper, by saying that a vector-valued or matrix-valued function belongs to a function space, we mean all the components belong to that space. The norm of a $d_1 \times d_2$ matrix $y$ is given by $|y|:= \sqrt{\text{Tr}(y y^{\mathrm{T}})}$. By $C>0$,  we denote a generic constant, which in particular does not depend on the discretization time step $h$ and possibly changes from line to line. When there is no ambiguity, we omit the argument $\omega$ in the proofs for simplicity of notations.

\section{Discrete-time approximation of the partially observed stochastic optimal control problem}\label{sectionconvergenceresult}

	In this section, we introduce our discrete-time scheme and the main convergence theorem. 
	We will start with the formulation \eqref{Ju1} for the optimal control problem with partial observation.
	Concretely, we consider the probability space $(\Omega, \mathcal{F}, \mathbb{Q})$, 
	in which $(B, Y)$ is a standard Brownian motion. Recall that $A$ is a compact separable metric space with distance $d_A$.
	Let $\cU$ denote the set of all admissible control processes $u$, i.e. $u$ is $A$-valued and progressively measurable w.r.t. the filtration $\mathbb{F}^Y$ generated by the Brownian motion $Y$. 
	Given $u \in \cU$, the process $X^u$ is defined by
	\begin{equation} \label{eq:dynamicXu}
		X_t^u = x_0 + \int_0^t b \left( s,X_s^u,u_s \right) ds + \int_0^t \sigma \left( s,X_s^u,u_s \right) dB_s,
		\quad t \in [0,T],
		\quad \mathbb{Q} \mbox{-a.s.}
	\end{equation}
	and the optimal control under partial observation is given by \eqref{Ju1}, i.e.
	\begin{equation} \label{eq:defV_main}
		V =\sup_{u\in\cU} J(u), 
		~\mbox{with}~
		J(u) := \mathbb{E}^{\mathbb{Q}} \left[ L_T^u \left(\int_0^T K(s,X_s^u,u_s)\,ds + G(X^u_T)\right) \right],
	\end{equation}
	where $L^u$ is defined in \eqref{Ltu}.
	To ensure that $X^u$ in \eqref{eq:dynamicXu} is well defined, and also for subsequent convergence analysis, 
	let us formulate the following conditions on the coefficients.

\begin{enumerate}
\renewcommand{\labelenumi}{($\mathbf{A1}$)}
\item
For the controlled dynamic, there exists a constant $C_1>0$ together with a continuity module $\rho$ such that the coefficients $\left(b,\sigma\right): [0,T]\times \mathbb{R}^d \times A \longrightarrow \mathbb{R}^{d\times (d\times m)}$ and $p: [0,T]\times \mathbb{R}^d \longrightarrow \mathbb{R}^{k}$ satisfy the uniform boundedness condition
$$
\left| b\left( t,x,a \right) \right| + \left| \sigma \left( t,x,a \right) \right| + \left| p \left(t,x\right) \right|\leq C_1,
$$
as well as the uniform continuity condition
\be\nonumber
\left| \left( b,\sigma \right) \left( t_1,x_1,a_1 \right) - \left( b,\sigma \right) \left(t_2,x_2,a_2 \right)\right| + \left| p \left(t_1,x_1 \right) - p \left(t_2,x_2\right) \right| \leq\rho\left( \left|t_1-t_2\right| \right) + C_1\left|x_1-x_2\right| + d_A(a_1, a_2) ,
\ee
for all $\left(t,t_1,t_2,x,x_1,x_2,a,a_1,a_2\right)\in [0,T]^3 \times \left( \mathbb{R}^d \right)^3 \times A^3$.

\renewcommand{\labelenumi}{($\mathbf{A2}$)}
\item
The instantaneous reward function $K : [0,T]\times \mathbb{R}^d \times A \longrightarrow \mathbb{R}$ and the terminal reward function $ G : \mathbb{R}^d \longrightarrow \mathbb{R}$ are continuous, and have the exponential growth for a constant $C_2 > 0$:
\begin{equation}\nonumber
\left|K\left(t,x,a\right)\right| + \left|G\left(x\right)\right|\leq C_2 \left( 1+ e^{C_2\left|x\right|}\right).
\end{equation}

\end{enumerate}

\subsection{Discrete-time approximation schemes}

	Based on the partially observable control problem~\eqref{eq:dynamicXu} and~\eqref{eq:defV_main}, 
	we consider the following discrete-time approximation schemes.

	\vspace{0.5em}

	For each $n\ge 1$, let us denote $h:=T/n$, and $t_k:= kh$, $\mathbb{T}_h := \left\{ t_0, t_1, \cdots , t_n \right\}$.
	On a fixed probability space $(\Omega^h,\mathscr{F}^h,\mathbb{Q}^h)$,
	equipped with some independent random variables $\left\{\eta^h_{i}, U_{i}^h \right\}_{1 \leq i \leq n}$,
	we introduce the discrete-time filtration 
	\begin{equation}\label{defscrFh}
		\mathbb{F}^h = \left\{ \mathscr{F}_{i}^h := \sigma \left( \eta^h_{m} \,,\, U_{m}^h : 1 \leq m \leq i \right) \right\}_{0 \leq i \leq n}.
	\end{equation}
	In above, $\eta^h_i$ is a $\mathbb{R}^k$-valued random variable, $U^h_i$ is a $[0,1]$-valued random variable with uniform distribution.
	Moreover, $\eta^h_i$ is independent of $U^h_i$.
	Namely, $(\eta^h_i)_{1 \le i \le n}$ will be used to define discrete-time observation process process $Y^h$,
	and $(U^h_i)_{1 \le i \le n}$ will be used to defined discrete-time controlled process $X^h$.

\subsubsection{Discrete observation $Y^h$, control $u^h$ and signal $X^{h,u^h}$}

	For each $h > 0$, let us define $Y^h$ by
	\begin{equation}\label{defYhi}
		Y^h_0=0, \quad\quad Y_{i+1}^h := Y_i^h + \eta^h_{i+1}, \quad i = 0, 1, \cdots, n-1.
	\end{equation}
	Then a discrete-time $A$-valued control process $\{u^h_i\}_{0\leq i \leq n}$ is admissible if it is adapted to the filtration of $Y^h$, i.e. 
	$u^h_i \in \sigma(Y^h_0, \cdots, Y^h_i)$, for each $i = 0, \cdots, n$.
	Let us denote by $\cU_h$ the collection of all admissible discrete-time control processes with parameter $h > 0$.

	Next, given a kernel function $H_h: \mathbb{T}_h \times \mathbb{R}^d \times A \times [0,1] \rightarrow \mathbb{R}^d $, we define the discrete-time signal process $\{X^{h,u^h}_i\}_{0\leq i \leq n}$ by
	\begin{equation}\label{defXhuhi}
		X_{0}^{h,{u^h}}=x_0,
		\quad
		X_{i+1}^{h,{u^h}} := X_i^{h,{u^h}} + H_h \left( t_i, X_i^{h,{u^h}}, u^h_i, U_{i+1}^h \right) ,
		\quad  i = 0, \cdots, n-1.
	\end{equation}

\begin{assumption}\label{definitionYhi}

	There exist constants $C > 0$, $C(c) > 0$ for each $c > 0$, independent of $h$, such that the following holds.

	\noindent $\mathrm{(i)}$ For each $i =  1, \cdots, n$, the variable $\eta^h_i$ satisfies
	\begin{equation}\label{propertyetahuh}
		\mathbb{E}^{\mathbb{Q}^h} \left[ \eta^h_{i} \right] = 0,
		\quad\quad 
		\mathrm{Var}^{\mathbb{Q}^h} \left[ \eta^h_{i} \right] = h I_k,
		\quad\quad 
		\mathbb{E}^{\mathbb{Q}^h} \left[ \,\left| \eta^h_{i} \right|^3 \right] \leq C h^\frac{3}{2}.
	\end{equation}

	\noindent $\mathrm{(ii)}$ For each $\left( t_i, x, a \right) \in \mathbb{T}_h \times \mathbb{R}^d \times A$,
	let $U$ be a random variable in $(\Omega^h,\mathscr{F}^h,\mathbb{Q}^h)$ with uniform distribution on $[0,1]$, one has
\begin{subequations}\label{propertyxihuh}
\begin{numcases}{}
\displaystyle \label{xih1}
\quad \mathbb{E}^{\mathbb{Q}^h} \left[ H_h \left( t_i, x, a, U \right)  \right] = b\left(t_i,x,a\right)h, \\
\displaystyle \label{xih2}
\quad \mathrm{Var}^{\mathbb{Q}^h} \left[ H_h \left( t_i, x, a, U \right)  \right] = \sigma \sigma^\mathrm{T} \left( t_i,x,a\right) h, \\
\displaystyle \label{xih3}
\quad \mathbb{E}^{\mathbb{Q}^h} \left[ \left| H_h \left( t_i,x,a,U\right) \right|^3 \right] \leq C h^\frac{3}{2},\\
\displaystyle \label{xih4}
\quad \mathbb{E}^{\mathbb{Q}^h} \left[ e^{c \left| H_h \left( t_i,x,a,U\right) \right| } \right] \leq 1 + C(c)\, h ,\quad \mbox{for all}~ c>0.
\end{numcases}
\end{subequations}

\end{assumption}

We give two simple settings for the discrete observation and signal process as follows.
\begin{example}\label{example12}
In the one-dimensional case $d=k=1$, the kernel function $H_h$ and the random variables $\left\{ \eta_i^h \right\}_{1\leq i \leq n}$ can be chosen such that
$$
H_h \left( t_i, x, a, U^h_{i+1} \right) \cong b \left( t_i, x, a\right) h + \sigma \left( t_i, x, a\right) \sqrt{h} \varUpsilon_{i+1}^{h,2}, \quad\quad\quad \eta^h_{i+1} \cong \sqrt{h}\,\varUpsilon_{i+1}^{h,1},
$$
where ``$\cong$'' means that the two random variables have same distribution, and the random variables $\left\{ \varUpsilon_{i}^{h,1} \,,\, \varUpsilon_{i}^{h,2} \right\}_{1\leq i \leq n}$ are independent random variables on $(\Omega^h,\mathscr{F}^h,\mathbb{Q}^h)$ such that 
\begin{equation}\nonumber
\mathbb{Q}^h\left(\varUpsilon_{i}^{h,1} = \pm 1\right) = \mathbb{Q}^h \left(\varUpsilon_{i}^{h,2} = \pm 1\right) =\frac{1}{2}.
\end{equation}
Then one can verify that the conditions~\eqref{propertyetahuh} and~\eqref{propertyxihuh} are satisfied. 

Alternatively, we can also choose the random variables $\left\{ \varUpsilon_{i}^{h,1} \,,\, \varUpsilon_{i}^{h,2} \right\}_{1\leq i \leq n}$ as i.i.d. $\cN(0,1)$-distributed Gaussian random variables on $(\Omega^h,\mathscr{F}^h,\mathbb{Q}^h)$.
\end{example}

\begin{example}\label{example123}
	Let $d=k=1$, and the kernel function $H_h$ satisfy
	\be
		\quad\quad H_h \left( t_i, x_k , a, U \right) = \left\{
		\ba{rcl}
		\displaystyle \Delta x , &&\text{with probability $p^{i,k,a}_+$;  }\\[0.3cm]
		\displaystyle -\Delta x , &&\text{with probability $p^{i,k,a}_- $;  }\\[0.3cm]
		\displaystyle 0, &&\text{with probability $1- p^{i,k,a}_+ - p^{i,k,a}_-$  ,}\ea
		\right.\label{defHh}
	\ee
	where
	$$
	p^{i,k,a}_\pm := \frac{1}{2} \left( \left( \sigma^2 + h b^2 \right) \left( t_i, x_k ,a\right) \frac{h}{\Delta x^2} \pm b \left( t_i, x_k ,a\right) \frac{h}{\Delta x} \right).
	$$	
	Setting $h=C^o \Delta x^2$ with $C^o := 1 / \left( \left\| b \right\|_\infty^2 + \left\| \sigma \right\|_\infty^2 \right)$,  we have that $1- p^{i,k,a}_+ - p^{i,k,a}_- \geq 0$. One can easily verify that $H_h$ defined in~\eqref{defHh} satisfy the conditions~\eqref{propertyxihuh}. 

	Notice that, with the above kernel function $H_h$, the discrete signal $X^{h,u^h}$ takes value in a countable set $\mathbb{Z}_h := \left\{ x_k = k \Delta x \right\}_{k\in \mathbb{Z}}$.
	In particular, this would induce an implementable numerical scheme for the approximation of the initial continuous time control problem (see more details in Section \ref{sectionFDtest}).

\end{example}

\subsubsection{Discrete-time stochastic optimal control problems with partial observation}

	We now introduce a sequence of discrete-time control problem, to approximate the value function $V$ in~\eqref{eq:defV_main}. 
	By~\eqref{Ltu}, let us first introduce the discrete Radon-Nikodym derivative.
	Recall that the Radon-Nikodym derivative $L^u$ is an exponential martingale and satisfies 
	\begin{equation}\label{Ltu1}
		L_t^u = 1+ \int_0^t p \left(s,{X}_s^{u}\right) {L}_s^{u}\, d Y_s.
	\end{equation}
	Let us define $L^{h,u^h}$ by the Euler type scheme for~\eqref{Ltu1}:
	\begin{equation}\label{Lhuhi}
		L^{h,u^h}_{0}=1,
		\quad\quad 
		L^{h,u^h}_{{i+1}} = L^{h,u^h}_{{i}} + L^{h,u^h}_{{i}} p \left( t_i,X_i^{h,{u^h}} \right) \eta^h_{i+1}, 
		\quad i = 0, 1, \cdots, n-1.
	\end{equation}

Then we use the above discrete processes~\eqref{defYhi},~\eqref{defXhuhi} and~\eqref{Lhuhi} to define the discrete-time reward functional
\be\label{defJhuh}
J_h\left(u^h\right) := \mathbb{E}^{\mathbb{Q}^h} \left[ L^{h,u^h}_n \left(\sum_{i=0}^{n-1} K \left( t_i,X_i^{h,{u^h}},u_i^h \right) h + G \left(X_n^{h,{u^h}} \right) \right) \right],
\ee
and define the corresponding value
\begin{equation}\label{defVh}
V_h=\sup_{u^h\in\cU_h} J_h \left(u^h\right).
\end{equation}

\subsection{Main convergence theorem}

	Let us now provide the main convergence result of the paper, whose proof will be reported in Section~\ref{sectionproofconvergence}.

	\begin{theorem}\label{ConvergenceTheorem}
		Let the assumptions $(\mathbf{A1})$ - $(\mathbf{A2})$, and Assumption \ref{definitionYhi} hold true. Then we have
		$$
			\lim_{h\rightarrow 0} V_h = V.
		$$
	\end{theorem}

	\begin{remark}\label{explain1}	
		To the best of our knowledge, Theorem~\ref{ConvergenceTheorem} provides a first convergence result for the approximation of the general partially observed control problem.
		Notice that, in the recent literature of the machine learning based numerical methods for control problems (see e.g. \cite{han2016deep, han2018solving}, etc.), 
		the formulation of the problem and the heuristic discussion are usually in continuous time, but the numerical algorithms are generally in discrete-time.
		Our convergence result in Theorem \ref{ConvergenceTheorem} would provide a general approach for the convergence analysis of these methods.

		It would also be more interesting to obtain a convergence rate of the error $\left| V_h-V \right|$ as $h\rightarrow 0$,
		for which we hope to investigate in the future.
	\end{remark}

\section{An implementable scheme and numerical test}\label{sectionFDtest}

	In the setting of Example \ref{example123}, with the choice of kernel function in~\eqref{defHh},
	the discrete signal and observation take values in a countable set. 
	Based on this, we obtain an implementable numerical scheme via the dynamic programming principle. 
	We will also conduct some numerical experiment for a partially observed control problem in a linear-quadratic setting.

\subsection{An implementable numerical scheme}

	For the increment of the discrete observation, let us set $\left\{ \eta_{i}^{h} \right\}_{1\leq i \leq n}$ to be independent random variables such that $\mathbb{Q}^h \left(\eta_{i}^{h} = \pm \sqrt{h} \right) =\frac{1}{2}$ as in the first case of Example~\ref{example12}. Then we can choose small enough $h>0$ satisfying $\left| \eta_{i+1}^h \right| < \left(\left\| p \right\|_\infty \right)^{-1}$, so that we have $L^{h,u^h}_i>0$ for each $i=0,1,\cdots,n$.

	The kernel function $H_h$ is given by~\eqref{defHh}.
	In particular, the process $X^{h, u^h}$ takes value in the discrete space $\{k \Delta x ~: k \in \mathbb{Z}\}$.
	Let us define $\mathbb{P}^{h, u^h}$ by
	$d {\mathbb{P}^{h,{u^h}}}   := {L^{h, {u^h}}_{n}}d \mathbb{Q}^h$ and consider the conditional distribution of ${X}^{h,{u^h}}_{l}$ given the past observation 
$$
{\mathscr{F}}^{{ Y}^h}_{l}=\sigma\left( {Y^h_0}, {Y^h_{1}}, \cdots, {Y^h_{l}} \right)
$$ 
under the probability ${ \mathbb{P}^{h,{u^h}}}$, i.e. the discrete filter process 
$$
{\mu^{h,u^h}_{l}} := \mathscr{L}^{\mathbb{P}^{h,{u^h}}} \left( {X}^{h,u^h}_{l} \,\Big|\, \mathscr{F}^{Y^h}_{l} \right),\quad\quad l=0,1,\cdots,n.
$$ 
Then the discrete reward~\eqref{defJhuh} can be rewritten as
\be\label{defJhuhnew}
J_h\left(u^h\right) := \mathbb{E}^{\mathbb{Q}^h} \left[ \lambda_n^{h,u^h} \times \left( \sum_{i=0}^{n-1} \mu^{h,u^h}_i \left( K \left( t_i,\,\cdot\,,u_i^h \right) \right) h + \mu^{h,u^h}_n \left(G \right) \right) \right],
\ee
where $\lambda_{0}^{h,u^h}=1$, and for $l=0,1,\cdots,n-1$,
\be\nonumber
\displaystyle \lambda_{l+1}^{h,u^h} &:=&\displaystyle \mathbb{E}^{\mathbb{Q}^h} \left[ {L^{h, {u^h}}_{l+1}} \,\Big|\, \mathscr{F}^{Y^h}_{l+1} \right] = \mathbb{E}^{\mathbb{Q}^h} \left[ L^{h, u^h}_{l} \left\{ 1 + p\left(t_{l}, X^{h,u^h}_{l} \right) \eta_{l+1}^h \right\} \,\Big|\, \mathscr{F}^{Y^h}_{l+1} \right] \\\nonumber
&=&\displaystyle \mathbb{E}^{\mathbb{Q}^h} \left[ L^{h, u^h}_{l} \,\Big|\, \mathscr{F}^{Y^h}_{l+1} \right] +  \mathbb{E}^{\mathbb{Q}^h} \left[ L^{h, u^h}_{l} p\left(t_{l}, X^{h,u^h}_{l} \right)  \,\Big|\, \mathscr{F}^{Y^h}_{l+1} \right] \eta_{l+1}^h.
\ee
Since $Y^h_{l+1}-Y^h_{l}$ is independent with $\mathscr{F}^h_{l}$, it follows that
\be\nonumber
\displaystyle \lambda_{l+1}^{h,u^h} &=&\displaystyle \mathbb{E}^{\mathbb{Q}^h} \left[ L^{h, u^h}_{l} \,\Big|\, \mathscr{F}^{Y^h}_{l} \right] +  \frac{ \mathbb{E}^{\mathbb{Q}^h} \left[ L^{h, u^h}_{l} p\left(t_{l}, X^{h,u^h}_{l} \right)  \,\Big|\, \mathscr{F}^{Y^h}_{l} \right] }{\mathbb{E}^{\mathbb{Q}^h} \left[ L^{h, u^h}_{l} \,\Big|\, \mathscr{F}^{Y^h}_{l} \right]} \cdot\mathbb{E}^{\mathbb{Q}^h} \left[ L^{h, u^h}_{l} \,\Big|\, \mathscr{F}^{Y^h}_{l} \right] \cdot \eta_{l+1}^h \\\nonumber
&=&\displaystyle \lambda_{l}^{h,u^h} + \mathbb{E}^{\mathbb{P}^{h,u^h}} \left[ p\left(t_{l}, X^{h,u^h}_{l} \right)  \,\Big|\, \mathscr{F}^{Y^h}_{l} \right] \cdot \lambda_{l}^{h,u^h} \cdot \eta_{l+1}^h = \lambda_{l}^{h,u^h} \left[ 1 + \mu^{h,u^h}_{l} \left( p\left(t_{l}, \cdot \right) \right) \cdot \eta_{l+1}^h \right],
\ee
which indicates that $\lambda_{l+1}^{h,u^h}$ is determined by the past data of discrete filter and observation $\left\{ \mu_i^{h,u^h}, \eta^h_{i+1} \right\}_{0\leq i \leq l}\,$.

Thus for solving the control problem corresponding to~\eqref{defJhuhnew}, the key is to study the dynamic of the discrete filter process $\mu^{h,u^h}$. Note that
\be\nonumber
\displaystyle {\mu^{h,u^h}_{l+1,k}} &:=&\displaystyle {\mathbb{P}^{h,{u^h}}} \left( {X}^{h,u^h}_{l+1} = x_k \,\Big|\, \mathscr{F}^{Y^h}_{l+1} \right) = \frac{ \mathbb{E}^{\mathbb{Q}^h} \left[ {L^{h, {u^h}}_{l+1}} \cdot \mathbb{I}_{\,{X}^{h,u^h}_{l+1} = x_k} \,\Big|\, \mathscr{F}^{Y^h}_{l+1} \right] }{\mathbb{E}^{\mathbb{Q}^h} \left[ {L^{h, {u^h}}_{l+1}} \,\Big|\, \mathscr{F}^{Y^h}_{l+1} \right]} = \frac{ \mathbb{E}^{\mathbb{Q}^h} \left[ {L^{h, {u^h}}_{l+1}} \cdot \mathbb{I}_{\,{X}^{h,u^h}_{l+1} = x_k } \,\Big|\, \mathscr{F}^{Y^h}_{l+1} \right] }{\lambda_{l+1}^{h,u^h}} \\\label{derivemul1}
&=&\displaystyle \frac{ \mathbb{E}^{\mathbb{Q}^h} \left[ {L^{h, {u^h}}_{l+1}} \cdot \mathbb{I}_{\,{X}^{h,u^h}_{l+1} = x_k } \,\Big|\, \mathscr{F}^{Y^h}_{l+1} \right] }{\lambda_{l}^{h,u^h} \left[ 1 + \mu^{h,u^h}_{l} \left( p\left(t_{l}, \cdot \right) \right) \cdot \eta_{l+1}^h \right]} = \frac{ \mathbb{E}^{\mathbb{Q}^h} \left[ L^{h, {u^h}}_{l} \left\{ 1+ p\left( t_l,X^{h,u^h}_l \right)\eta^h_{l+1} \right\} \cdot \mathbb{I}_{\,{X}^{h,u^h}_{l+1} = x_k } \,\Big|\, \mathscr{F}^{Y^h}_{l+1} \right] }{ \mathbb{E}^{\mathbb{Q}^h} \left[ {L^{h, {u^h}}_{l}} \,\Big|\, \mathscr{F}^{Y^h}_{l} \right] \times \left[ 1 + \mu^{h,u^h}_{l} \left( p\left(t_{l}, \cdot \right) \right) \cdot \eta_{l+1}^h \right]}.
\ee
For the numerator of the last term in~\eqref{derivemul1}, we have
\be\nonumber
&&\displaystyle \mathbb{E}^{\mathbb{Q}^h} \left[ L^{h, {u^h}}_{l} \left\{ 1+ p\left( t_l,X^{h,u^h}_l \right)\eta^h_{l+1} \right\} \cdot \mathbb{I}_{\,{X}^{h,u^h}_{l+1} = x_k } \,\Big|\, \mathscr{F}^{Y^h}_{l+1} \right] \\\nonumber
&=&\displaystyle \mathbb{E}^{\mathbb{Q}^h} \left[ L^{h, {u^h}}_{l} \left\{ 1+ p\left( t_l,x_{k-1} \right)\eta^h_{l+1} \right\} \cdot \mathbb{I}_{\,{X}^{h,u^h}_{l} = x_{k-1} } \cdot \mathbb{I}_{\,H_h \left( t_l, x_{k-1} , u^h_l, U^h_{l+1} \right) = \Delta x } \,\Big|\, \mathscr{F}^{Y^h}_{l+1} \right] \\\nonumber
&&\displaystyle + \mathbb{E}^{\mathbb{Q}^h} \left[ L^{h, {u^h}}_{l} \left\{ 1+ p\left( t_l,x_{k} \right) \eta^h_{l+1} \right\} \cdot \mathbb{I}_{\,{X}^{h,u^h}_{l} = x_{k} } \cdot \mathbb{I}_{\,H_h \left( t_l, x_{k} , u^h_l, U^h_{l+1} \right) = 0 } \,\Big|\, \mathscr{F}^{Y^h}_{l+1} \right] \\\label{last2}
&&\displaystyle + \mathbb{E}^{\mathbb{Q}^h} \bigg[ L^{h, {u^h}}_{l} \left\{ 1+ p\left( t_l,x_{k+1} \right) \eta^h_{l+1} \right\} \cdot \mathbb{I}_{\,{X}^{h,u^h}_{l} = x_{k+1} } \cdot \mathbb{I}_{\,H_h \left( t_l, x_{k+1} , u^h_l, U^h_{l+1} \right) = -\Delta x } \,\Big|\, \mathscr{F}^{Y^h}_{l+1} \bigg].
\ee
For the first term of the right hand side of the above equality, it holds that
\be\nonumber
&& \displaystyle \mathbb{E}^{\mathbb{Q}^h} \left[ L^{h, {u^h}}_{l} \left\{ 1+ p\left( t_l,x_{k-1} \right) \eta^h_{l+1} \right\} \cdot \mathbb{I}_{\,{X}^{h,u^h}_{l} = x_{k-1} } \cdot \mathbb{I}_{\,H_h \left( t_l, x_{k-1} , u^h_l, U^h_{l+1} \right) = \Delta x } \,\Big|\, \mathscr{F}^{Y^h}_{l+1} \right] \\\nonumber
&=&\displaystyle \mathbb{E}^{\mathbb{Q}^h} \Bigg[ L^{h, {u^h}}_{l} \left\{ 1+ p\left( t_l,x_{k-1} \right) \eta^h_{l+1} \right\} \cdot \mathbb{I}_{\,{X}^{h,u^h}_{l} = x_{k-1} } \cdot \mathbb{E}^{\mathbb{Q}^h} \left[ \mathbb{I}_{\,H_h \left( t_l, x_{k-1} , u^h_l, U^h_{l+1} \right) = \Delta x } \,\Big|\, \mathscr{F}^{Y^h}_{l+1} \vee \mathscr{F}^h_{l} \right] \,\Bigg|\, \mathscr{F}^{Y^h}_{l+1} \Bigg] \\\nonumber
&=&\displaystyle \mathbb{E}^{\mathbb{Q}^h} \left[ L^{h, {u^h}}_{l} \left\{ 1+ p\left( t_l,x_{k-1} \right) \eta^h_{l+1} \right\} \mathbb{I}_{\,{X}^{h,u^h}_{l} = x_{k-1} } p^{l,k-1,u^h_l}_+ \,\Bigg|\, \mathscr{F}^{Y^h}_{l+1} \right] \\\nonumber
&=&\displaystyle \mathbb{E}^{\mathbb{Q}^h} \left[ L^{h, {u^h}}_{l}  \mathbb{I}_{\,{X}^{h,u^h}_{l} = x_{k-1} } \Bigg| \mathscr{F}^{Y^h}_{l+1} \right] \left\{ 1+ p\left( t_l,x_{k-1} \right) \eta_{l+1}^h \right\} p^{l,k-1,u^h_l}_+ = \mathbb{E}^{\mathbb{Q}^h} \left[ L^{h, {u^h}}_{l}  \mathbb{I}_{\,{X}^{h,u^h}_{l} = x_{k-1} } \Bigg| \mathscr{F}^{Y^h}_{l} \right] \left\{ 1+ p\left( t_l,x_{k-1} \right) \eta_{l+1}^h \right\} p^{l,k-1,u^h_l}_+,
\ee
where $\mathscr{F}^h_i$ is defined in~\eqref{defscrFh} and $p^{l,k-1,u^h_l}_+$ is defined in~\eqref{defHh}. Performing the similar calculation for the last two terms in~\eqref{last2}, we get
\be\nonumber
\mathbb{E}^{\mathbb{Q}^h} \left[ {L^{h, {u^h}}_{l+1}} \cdot \mathbb{I}_{\,{X}^{h,u^h}_{l+1} = x_k } \,\Big|\, \mathscr{F}^{Y^h}_{l+1} \right] 
&=&\displaystyle \mathbb{E}^{\mathbb{Q}^h} \left[ L^{h, {u^h}}_{l}  \cdot \mathbb{I}_{\,{X}^{h,u^h}_{l} = x_{k-1} } \,\Bigg|\, \mathscr{F}^{Y^h}_{l} \right] \left\{ 1+ p\left( t_l,x_{k-1} \right) \eta_{l+1}^h \right\} \cdot p^{l,k-1,u^h_l}_+ \\\nonumber
&&\displaystyle + \mathbb{E}^{\mathbb{Q}^h} \left[ L^{h, {u^h}}_{l}  \cdot \mathbb{I}_{\,{X}^{h,u^h}_{l} = x_{k} } \,\Bigg|\, \mathscr{F}^{Y^h}_{l} \right] \left\{ 1+ p\left( t_l,x_{k} \right) \eta_{l+1}^h \right\} \cdot \left( 1- p^{l,k,u^h_l}_+ - p^{l,k,u^h_l}_- \right) \\\label{numerator}
&&\displaystyle + \mathbb{E}^{\mathbb{Q}^h} \left[ L^{h, {u^h}}_{l}  \cdot \mathbb{I}_{\,{X}^{h,u^h}_{l} = x_{k+1} } \,\Bigg|\, \mathscr{F}^{Y^h}_{l} \right] \left\{ 1+ p\left( t_l,x_{k+1} \right) \eta_{l+1}^h \right\} \cdot p^{l,k+1,u^h_l}_-.
\ee
Combining~\eqref{derivemul1} with~\eqref{numerator}, it holds that
\be\nonumber
\displaystyle {\mu^{h,u^h}_{l+1,k}} &=&\displaystyle \frac{1}{1 + \left[ \sum\limits_{m\in\mathbb{Z} } p\left( t_l , x_m \right)\cdot \mu_{l,m}^{h,u^h} \right] \cdot {\eta}^h_{l+1}} \times \Bigg\{ \left[ 1 + p\left( t_l, x_{k-1} \right){\eta}^h_{l+1} \right] \cdot p^{l,k-1,u^h_l}_+ \cdot {\mu^{h,u^h}_{l,k-1}}
 \\\nonumber
&&\displaystyle  \quad\quad\quad\quad\quad\quad\quad\quad\quad\quad\quad\quad\quad\quad\quad + \left[ 1 + p\left( t_l , x_k \right){\eta}^h_{l+1} \right] \cdot \left( 1 - p^{l,k,u^h_l}_+ - p^{l,k,u^h_l}_- \right) \cdot { \mu^{h,u^h}_{l,k}} \\\nonumber
&&\displaystyle  \quad\quad\quad\quad\quad\quad\quad\quad\quad\quad\quad\quad\quad\quad\quad +\left[ 1 + p\left( t_l, x_{k+1} \right){\eta}^h_{l+1} \right] \cdot p^{l,k+1,u^h_l}_- \cdot { \mu^{h,u^h}_{l,k+1}} \Bigg\} .
\ee
The above implies that, for some function $\mathbf{\Xi}^h $, one has
\begin{equation}\label{markov}
{\mu^{h,u^h}_{l+1}} = \mathbf{\Xi}^h \left( l, u^h_l, \mu^{h,u^h}_l,{\eta}^h_{l+1} \right).
\end{equation}

We denote by ${V}_h \left( t_l, \mu \right)$ the discrete value starting from $t_l$ with the initial distribution ${\mu}$. In view of~\eqref{markov}, the discrete filter satisfies the Markov property and the flow property, which allows us to use DPP (see e.g.~\cite[Theorem 4.1]{MR957652}) to compute the value corresponding to~\eqref{defJhuhnew}, by an inductively backward way: for $0 \leq l \leq n-1$,
\begin{subequations}\label{FD}
\be\nonumber
\displaystyle {V}_h \left( t_l, \mu \right) &=&\displaystyle \sup_{a \in A} \mathbb{E}^{ {\mathbb{Q}^h} } \Big[ \left( 1 + \mu\left( p\left( t_l, \cdot \right) \right) {\eta}^h_{l+1} \right) \times \big\{ \mu\left( K \left( t_l, \cdot, a \right)\right) h + {V}_h \left( t_{l+1}, \mathbf{\Xi}^h \left( l, a, \mu, {\eta}^h_{l+1} \right) \right) \big\} \Big]\\\label{FD1}
&=&\displaystyle \sup_{a \in A} \mathbb{E}^{ {\mathbb{Q}^h} } \Big[ \mu\left( K \left( t_l, \cdot, a \right)\right) h +\left( 1 + \mu \left[ p\left( t_l, \cdot \right) \right] \eta^h_{l+1} \right) {V}_h \left( t_{l+1}, \mathbf{\Xi}^h \left( l, a, \mu, \eta^h_{l+1} \right) \right) \Big],
\ee
with the terminal condition
\begin{equation}\label{FD2}
{V}_h \left( t_n, \mu \right) = \mu \left( G \right) = \int G(x) \mu(dx).
\end{equation}
\end{subequations}
Notice that $\mu$ has support in the discrete space $\{k \Delta x ~: k \in \mathbb{Z} \}$.

\subsection{Numerical test}\label{NumLQsec}

In this subsection, we illustrate our theoretical result by a simple numerical example with a linear quadratic structure in one dimension $d=k=1$, so that the reference value of the problem can be computed explicitly. Concretely, we set 
$$
A=\mathbb{R},\quad b(t,x,a)=a,\quad \sigma(t,x,a)\equiv 1,\quad p(t,x)=x,\quad K(t,x,a)=a^2,\quad G(x)=x^2.
$$ 
In this way, the linear controlled system is given by
\begin{subequations}\label{LQ}
\begin{equation}
X_t^u = x_0 + \int_0^t u_s \, ds + B_t, \quad\quad Y_t= \int_0^t X_s^u\, ds + W_t,\label{LQsde}
\end{equation}
where $(B,W)$ is an $\mathbb{R}^2$-valued standard Brownian motion on $\left(\Omega,\mathscr{F},\mathbb{P}\right)$, with the quadratic reward
\begin{equation}\label{LQcost}
J(u)= \mathbb{E}^\mathbb{P} \left[ \int_0^T \left| u_s \right|^2 ds + \left|X^u_T\right|^2 \right] = \mathbb{E}^{\mathbb{Q}} \left[ L_T^u \left(\int_0^T \left| u_s \right|^2 ds + \left|X^u_T\right|^2 \right) \right],
\end{equation}
\end{subequations}
where $\frac{d \mathbb{Q}}{d \mathbb{P}} := \left( L^u_T \right)^{-1}$ and $L_t^u :=\exp\left[ \int_0^t X_s^u \,d Y_s - \frac{1}{2} \int_0^t \left| X_s^u \right|^2 d s \right]$ such that $(B,Y)$ is an $\mathbb{R}^2$-valued standard Brownian motion on $\left(\Omega,\mathscr{F},\mathbb{Q}\right)$.

\begin{figure}[htp]\centering
    \includegraphics[width=0.7\linewidth]{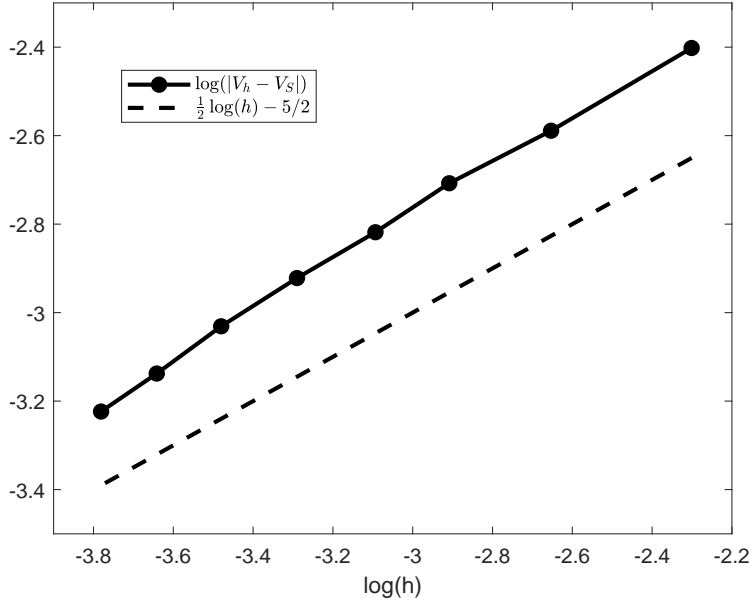}
  \caption{Approximation errors for Example~\eqref{LQ} using the numerical scheme~\eqref{FD}.}
  \label{figLQ}
\end{figure}

Following Bensoussan~\cite[Theorem 2.4.1]{MR1191160}, one obtains the explicit solution of the optimal control for~\eqref{LQ}:
$$
u^*_t= - \Pi_t\,\hat{y}_t,
$$ 
where 
$$
d \hat{y}_t = -\left( \Pi_t + P_t \right) \hat{y}_t\,dt + P_t\,dY_t,\quad\quad \hat{y}_0= x_0,
$$
with $\Pi$ being the solution to the backward Riccati equation
$$
\frac{d \Pi_t}{dt}=\Pi_t^2,\quad\quad \Pi_T=1,
$$
and $P$ being the solution to the forward variance equation
$$
\frac{d P_t}{dt}= 1- P_t^2,\quad\quad P_0=  0.
$$

We implement the numerical scheme~\eqref{FD} with $x_0=0$, $T=0.1$, and $h=\Delta x^2 /10$. The numerical results are presented on Figure~\ref{figLQ}, where the solid line displays the monotonic trend of the log errors $\log(| V_h - V_S|)$ against the log values of the time mesh size $\log(h)$, and the dotted line is the reference line with slope $1/2$. We observe that the empirical rate of convergence is about $1/2$. It should be pointed out that the error in Figure~\ref{figLQ} consists of two sources. The first source of error comes from the time discretization of the scheme~\eqref{FD}. The second source of error comes from the truncation of computational space area, and the linear interpolation of the measure-valued function $\mathbf{\Xi}^h$ which is stated in Remark~\ref{R31}. In particular, when we set the time mesh size to be $h=1.66\times 10^{-4}$ and use $12$ spatial grids from $[-0.49,0.49]$, the error range is no more than $0.6\%$, which indicates that our scheme works well to approximate the exact value.

\begin{remark}\label{R31}
To implement algorithm~\eqref{FD}, one also needs to discretize the measure $\mu$. One alternative way is that, fixing $M\in\mathbb{Z}_+$, we discretize the measure to be the form of $\mu(x_k)=m/M$ for some $m=0,1,\cdots,M$ such that $\sum_{k\in\mathbb{Z}} \mu(x_k) =1$. Moreover, if we use finite spatial girds, there are finite number of measures $\{\mu^j \}_{1 \leq j \leq K}$ to be considered.  Then at each time step $t_l$, we can use proper linear combination to approximate the measure-valued update function $\mathbf{\Xi}^h ( l, a, \mu^j, \eta^h_{l+1} ) \approx \sum_{i=1}^K \alpha_i \,\mu^i$. And the value can be approximated by the first-order linear expansion ${V}_h ( t_{l+1}, \mathbf{\Xi}^h ( l, a, \mu, \eta^h_{l+1} )) \approx \sum_{i=1}^K \alpha_i \,{V}_h ( t_{l+1}, \mu^i )$.
\end{remark}

\section{Partially observed optimal control: different formulations}\label{sectionformulation}

	In preparation of the proof of main convergence result in Theorem \ref{ConvergenceTheorem}, we follow~\cite{MR957652} to introduce here strong, weak, and relaxed formulations of the stochastic optimal control problem under partial observation, and revisit the corresponding equivalence result.
	
	Let us first introduce some canonical spaces.

\begin{enumerate}

	\renewcommand{\labelenumi}{(i)}
	\item Let $\cD^d := D([0,T], \mathbb{R}^d)$
	be the space of all c\`adl\`ag  $\mathbb{R}^d$-valued paths on $[0,T]$ equipped with its canonical filtration $\mathscr{D}^d_t$.
	Notice that $\cD^d$ is a Polish space under the Skorokhod topology.

	\renewcommand{\labelenumi}{(ii)}
	\item 
	Let $\mathcal{C}^d :=C([0,T],\mathbb{R}^d) $ 
	be the space of all continuous  $\mathbb{R}^d$-valued paths on $[0,T]$ equipped with its canonical filtration $\mathscr{C}^d_t$.
	Notice that $\cC^d$ is a Polish space under the uniform convergence topology, which is also a closed subset of $\cD^d$ (under the Skorokhod topology).

	\renewcommand{\labelenumi}{(iii)}
	\item 
	Let $\mathbf{M}([0,T]\times A)$ denote the space of all finite positive measures $q$ on $[0,T]\times A$ such that $q(dt , A)=dt$, which is a Polish space equipped 
	with the weak convergence topology. We will use its subset $V$ whose elements admit disintegration, i.e.,
	$$
		V:=\left\{ q\in \mathbf{M}([0,T]\times A) :  q(dt,da)=q(t,da)dt  ~\text{s.t.}~ \int_A q(t,da) = 1, ~\forall t \right\}.
	$$
	Notice that $V$ is closed under weak convergence topology and hence is also a Polish space. We define the filtration on $V$ by 
	$$
	\mathscr{V}_t:=\sigma \left\{ \int_0^s \varphi(r,a)\,q(dr,da),\quad s\leq t,\quad \varphi\in C_b([0,T]\times A)\right\}.
	$$
	 In particular, $\mathscr{V}_T$ is the Borel $\sigma$-field of $V$. The set of measurable functions $v$ from $[0,T]$ into $A$ is embedded in $V$ in a natural way by defining the atomic measure $q^v(ds,da):=\delta_{v(s)}(da)ds$, where $\delta_z$ is the Dirac measure at $z\in A$. The set of the atomic measures is denoted by $V^0$.
\end{enumerate}

	Next, we define the enlarged canonical space
	\begin{equation}\nonumber
		\overline{\Omega}:=\mathcal{D}^d \times \mathcal{D}^1 \times \mathcal{D}^k \times V,
	\end{equation}
	with the canonical filtration $(\overline{\mathscr{F}}_t)_{t \in [0,T]}$ defined by
	$$
		\overline{\mathscr{F}}_t:= \mathscr{D}_t^d \otimes \mathscr{D}_t^1 \otimes \mathscr{D}_t^k \otimes \mathscr{V}_t.
	$$
	The canonical process is denoted by $(X,L,Y,q)$. 
	We will also use the sub-filtration $(\overline{\mathscr{G}}_t)_{t \in [0,T]}$ generated by the canonical process $(Y,q)$, i.e.
	$$
		\overline{\mathscr{G}}_t:= \mathscr{D}_t^k \otimes \mathscr{V}_t .
	$$
	
	\begin{remark}
		To study the continuous time control problem with partial observation \eqref{eq:defV_main}, it is natural to use the canonical space $\cC^d \x \cC^1\x  \cC^k \x V$.
		We nevertheless use the Skorokhod space of c\`adl\`ag paths as canonical space to include discrete-time problems.
		In fact, in the latter part of the paper, we will consider a discrete-time process as a continuous-time process with piece-wise constant paths.
		
	\end{remark}

\subsection{Strong formulation}

	The strong formulation is that given in Introduction or Section \ref{sectionconvergenceresult}.
	We nevertheless recall it here and then reformulate it equivalently on the canonical space $\overline \Omega$.
	Let us consider the canonical space
\begin{equation}\nonumber
\Omega:=\mathcal{C}^m \times \mathcal{C}^k,
\end{equation}
equipped with its Borel $\sigma$-algebra $\mathscr{F}:=\mathscr{B}(\Omega)$ and canonical element $(B,Y)$. Let $\mathbb{F}:=\left\{ \mathscr{F}_t \right\}_{0\leq t \leq T}$ and $\mathbb{G}:=\left\{ \mathscr{G}_t \right\}_{0\leq t \leq T}$ be the filtrations on $(\Omega,\mathscr{F})$ defined by
\begin{equation}\nonumber
\mathscr{F}_t := \sigma\left\{ \left(B_s,Y_s \right):s\in[0,t]\right\},\quad\text{and}\quad \mathscr{G}_t := \sigma\left\{Y_s:s\in[0,t]\right\},\quad t\in[0,T].
\end{equation}
	Let us denote by $\mathbb{Q}$ the probability measure on $(\Omega,\mathscr{F})$ under which $(B,Y)$ is a standard $\mathbb{R}^{m+k}$-valued Brownian motion. 
	Further, recall that $A$ is a given compact separable metric space, 
	$\cU$ denotes the space of all admissible control processes  $u = \left\{u_t\right\}_{0\leq t \leq T}$,
	i.e. $u$ is $A$-valued and progressively measurable w.r.t. the filtration $\mathbb{G}$. 

	\vspace{0.5em}

	Then, given an admissible control $u\in \mathcal{U}$, the controlled signal SDE and the Radon-Nikodym derivative SDE, $\mathbb{Q}$-a.s.,
		\be
		 \left\{
		\ba{rcl}
		&&\displaystyle X_t^u = x_0 + \int_0^t b \left(s,X_s^u,u_s \right)ds + \int_0^t \sigma \left(s,X_s^u,u_s\right) dB_s\\[0.3cm]
		&&\displaystyle L_t^u = 1 + \int_0^t p \left(s,X_s^u\right) L^u_s \,dY_s,\ea
		\right.\label{eq:dynamicXL}
	\ee
	admit unique strong solution $\left( X^u, L^u \right)$.
	Then the strong formulation of the control problem with partial observation is given by
	\begin{equation}\label{Ju11}
		V_S := \sup_{u\in\mathcal{U}} J(u),
		\quad\mbox{with}\quad
		J(u)= \mathbb{E}^{\mathbb{Q}} \left[ L_T^u \left(\int_0^T K(s,X_s^u,u_s)\,ds + G(X^u_T)\right) \right].
	\end{equation}

Next we reformulate the control problem on the enlarged canonical space. For each $u\in \mathcal{U}$, we define the atomic measure $q^u(dt,da) = q^u(t,da) dt :=\delta_{u_t}(da)dt$. Then we introduce the set of strong control rules $\mathcal{R}_S$, which is a set of probability measures on the canonical space $\overline{\Omega}$ as follows:
\begin{equation}\nonumber
\mathcal{R}_S := \left\{ \overline{\mathbb{R}}_S = \mathbb{Q} \circ \left( X^u, L^u, Y, q^u \right)^{-1} :u\in\mathcal{U} \right\}.
\end{equation}

We can rewrite~\eqref{Ju11} as
\be\label{GammacR}
\displaystyle J \left(\overline{\mathbb{R}}_S\right) = \mathbb{E}^{\overline{\mathbb{R}}_S} \left[ \Gamma \right],
\ee
where the continuous function $\Gamma$ on $\overline{\Omega}$ is defined by
\begin{equation}\label{Gamma}
\Gamma(X,L,q) := L_T \bigg( \int_0^T \int_A K\left(s,X_s,a\right) q \left(s,da\right) ds + G\left(X_T\right) \bigg).
\end{equation}
Thus we can compute the value by
\begin{equation}\nonumber
V_S = \sup_{\overline{\mathbb{R}}_S \in\cR_S} \mathbb{E}^{\overline{\mathbb{R}}_S} \left[ \Gamma \right].
\end{equation}

Next we introduce two subsets of the strictly rules set $\cR_S$.
\begin{definition}
(i). We denote by $\cR^{s}_S$ the set of the strictly step rules $\overline{\mathbb{R}}^{s}_S$ on $\overline{\Omega}$ with
$$
\overline{\mathbb{R}}^{s}_S = \mathbb{Q} \circ \left( X^{u^{s,\kappa}}, L^{u^{s,\kappa}},\, Y\,,\, q^{u^{s,\kappa}} \right)^{-1},
$$ 
where ${u^{s,\kappa}} \in \mathcal{U}$ is a \textbf{step strategies} such that, with some $\kappa>0$,
$$
u^{s,\kappa}_t := w_i^{s,\kappa} \left(Y|_{[0,i\kappa]}\right),\quad\quad t\in[i\kappa,(i+1)\kappa),
$$
where $w_i^{s,\kappa}$ is a measurable functional on $C([0,i\kappa];\mathbb{R}^k)$.

(ii). Similarly, we denote by $\cR^{sd}_{S}$ the set of the strictly discrete step rules $\overline{\mathbb{R}}^{sd}_{S}$ on $\overline{\Omega}$ with
$$
\overline{\mathbb{R}}^{sd}_{S} = \mathbb{Q} \circ \left( X^{u^{sd,\kappa}}, L^{u^{sd,\kappa}},\, Y\,,\, q^{u^{sd,\kappa}} \right)^{-1},
$$ 
for some \textbf{discrete step strategies} ${u^{sd,\kappa}}$ with some $\kappa>0$, defined by
$$
u^{sd,\kappa}_t := {w}_i^{sd,\kappa} \left(Y_{r_j^i} \,,\, j\leq J_i\,,\,r_j^i \leq i\kappa \right),\quad\quad t\in[i\kappa,(i+1)\kappa),
$$
where ${w}_i^{sd,\kappa}$ is a uniformly Lipschitz continuous and bounded function on $\mathbb{R}^{k\times J_i}$ with $J_i\in\mathbb{N}_+$.
\end{definition}

The corresponding values are 
$$
V^s_S = \sup_{\overline{\mathbb{R}}^s_S\in\cR^s_S}\mathbb{E}^{\overline{\mathbb{R}}^s_S} \left[ \Gamma \right],\quad\quad\quad V^{sd}_{S} = \sup_{\overline{\mathbb{R}}^{sd}_{S} \in\cR^{sd}_{S}} \mathbb{E}^{\overline{\mathbb{R}}^{sd}_{S} } \left[ \Gamma \right].
$$

\subsection{A weak formulation}

As in the classical SDE theory, one can consider all possible probability spaces to define a weak solution of the controlled system~\eqref{strict}.

\begin{definition}\label{defwf} (Weak Control) We say that a term
\begin{equation}\nonumber
\gamma = \left( \Omega^\gamma, \mathscr{F}^\gamma, \mathbb{Q}^\gamma, \mathbb{F}^\gamma, \mathbb{G}^\gamma , X^\gamma, L^\gamma, B^\gamma, Y^\gamma, u^\gamma = \left\{ u^\gamma_t \right\}_{0 \leq t \leq T} \right)
\end{equation}
is a weak control if

\begin{enumerate}

\item
The $\left( \Omega^\gamma, \mathscr{F}^\gamma, \mathbb{Q}^\gamma \right)$ is a complete probability space equipped with two right continuous filtrations $\mathbb{F}^\gamma = \left\{ \mathscr{F}_t^\gamma \right\}_{0\leq t \leq T}$ and $\mathbb{G}^\gamma = \left\{ \mathscr{G}_t^\gamma \right\}_{0\leq t \leq T}$ such that $\mathscr{G}_t^\gamma \subseteq \mathscr{F}_t^\gamma$, for all $t\in[0,T]$;

\item
$(B^\gamma,Y^\gamma)$ is an $\mathbb{R}^{m+k}$-valued standard Brownian with respect to $\mathbb{F}^\gamma$ under $\mathbb{Q}^\gamma$; in addition, $Y^\gamma$ is adapted to $\mathbb{G}^\gamma$ and $B^\gamma$ is an $ \mathscr{G}_T^\gamma \vee \mathscr{F}_t^\gamma$ Brownian motion;

\item
$u^\gamma$ is an $A$-valued and $\mathbb{G}^\gamma$-adapted process;

\item
The $\mathbb{F}^\gamma$-adapted process $\left( X^\gamma, L^\gamma \right)$ satisfies that for all $t\in[0,T]$, $\mathbb{Q}^\gamma$-a.s.,
		\be
		 \left\{
		\ba{rcl}
		&&\displaystyle X_t^\gamma = x_0 + \int_0^t b \left(s,X_s^\gamma,u_s^\gamma \right)ds + \int_0^t \sigma \left(s,X_s^\gamma,u_s^\gamma\right) dB_s^\gamma, \\[0.3cm]
		&&\displaystyle L_t^\gamma = 1 + \int_0^t p \left(s,X_s^\gamma\right) L^\gamma_s\,dY^\gamma_s.\ea
		\right.\nonumber
	\ee

\end{enumerate}

\end{definition}

As in the strong formulation, for each weak control $\gamma$, we define the atomic measure $q^\gamma(dt,da) = q^\gamma(t,da) dt :=\delta_{u_t^\gamma}(da)dt$. Then the set of weak control rules is defined by
\begin{equation}\nonumber
\mathcal{R}_W := \left\{ \overline{\mathbb{R}}_W = \mathbb{Q}^\gamma \circ \left( X^\gamma, L^\gamma, Y^\gamma, q^\gamma \right)^{-1} :\text{$\gamma$ is a weak control} \right\},
\end{equation}
and the value of the weak formulation is defined by
\begin{equation}\nonumber
V_W = \sup_{\overline{\mathbb{R}}_W \in\cR_W} \mathbb{E}^{\overline{\mathbb{R}}_W} \left[ \Gamma \right].
\end{equation}

\begin{remark}\label{remarkmartingale}
We state two martingale properties for the weak rules $\overline{\mathbb{R}}_W$. By the It\^o formula, it follows that 

\begin{enumerate}
\item
for each $f\in C_b^{\infty}([0,T]\times \mathbb{R}^d)$, the process
\begin{equation}\nonumber
C_t^{f,X}:=f \left( t,X^\gamma_t \right) - f \left( 0,X^\gamma_0 \right)-\int_0^t \cL^{s,X^\gamma_s,u^\gamma_s} f \left( s,X^\gamma_s \right)ds
\end{equation}
is a $\mathscr{G}^\gamma_T \vee \mathscr{F}^\gamma_t$-martingale under $\mathbb{Q}^\gamma$, where the Fokker-Planck operator $\cL^{t,x,a}$ is defined by
\begin{equation}\label{cLB}
\cL^{t,x,a}f := f_t + b\left(t,x,a\right)\cdot Df + \frac{1}{2} \text{Tr} \left[ \sigma \sigma^\mathrm{T} \left(t,x,a\right) \times D^2f \right];
\end{equation}
\item
for each $g\in C_b^{\infty}([0,T]\times \mathbb{R} \times \mathbb{R}^k)$, the process
\begin{equation}\nonumber
D_t^{g,L}:=g \left( t,L^\gamma_t,Y^\gamma_t \right)-g \left( 0,L^\gamma_0,Y^\gamma_0 \right) - \int_0^t \cM^{s,X^\gamma_s,L^\gamma_s}g \left( s,L^\gamma_s,Y^\gamma_s \right)ds
\end{equation}
is a $\mathscr{F}^\gamma_t$-martingale under $\mathbb{Q}^\gamma$, where the Fokker-Planck operator $\cM^{t,x,l}$ is defined by
\begin{equation}\label{cML}
\cM^{t,x,l}g := g_t + \frac{1}{2}\text{Tr} \left[ \begin{pmatrix} pp^\mathrm{T} \left(t,x\right) l^2 & p^\mathrm{T} \left(t,x\right) l\,  \\ p \left(t,x\right) l & I_k \end{pmatrix} \times D^2g \right].
\end{equation}
\end{enumerate}

Thus the process
\begin{equation}\label{CtfB[x][z]}
C_t^{f}(X,q) :=f\left( t,X_t \right)-f \left( 0,X_0 \right)-\int_0^t \int_A \cL^{s,X_s,a}f\left( s,X_s \right)q\left(s,da\right)ds
\end{equation}
is a $\overline{\mathscr{G}}_T \vee \overline{\mathscr{F}}_t$-martingale under $\overline{\mathbb{R}}_W$; and the process
\begin{equation}\label{Dtg}
D_t^{g}(X,L,Y) := g \left(t,L_t,Y_t\right) - g\left(0,L_0,Y_0\right)-\int_0^t \cM^{s,X_s,L_s}g\left(s,L_s,Y_s\right)ds
\end{equation}
is a $\overline{\mathscr{F}}_t$-martingale under $\overline{\mathbb{R}}_W$.

\end{remark}

\subsection{A relaxed formulation}

Based on the martingale properties of the weak control in Remark~\ref{remarkmartingale}, we give the relaxed formulation. 

\begin{definition}\label{defrf}
(Relaxed Control Rules)
We denote by $\cR$ the set of relaxed control rules $\overline{\mathbb{R}}$, which are probability measures on $(\overline{\Omega},\overline{\mathscr{F}})$ such that $X_0=x_0$ and the followings hold.
\begin{enumerate}

\item
For each $f\in C_b^{\infty}([0,T]\times \mathbb{R}^d)$, the process $\left\{C_t^{f}(X,q)\right\}_{0\leq t\leq T}$ defined by~\eqref{CtfB[x][z]} is a $\left\{\overline{\mathscr{G}}_T \vee \overline{\mathscr{F}}_t \right\}_{0\leq t\leq T}$ -martingale under $\overline{\mathbb{R}}$;

\item
For each $g\in C_b^{\infty}([0,T]\times \mathbb{R} \times \mathbb{R}^k)$, the process $\left\{ D_t^{g}(X,L,Y) \right\}_{0\leq t\leq T}$ defined by~\eqref{Dtg} is a $\left\{\overline{\mathscr{F}}_t \right\}_{0\leq t\leq T}$ -martingale under $\overline{\mathbb{R}}$.

\end{enumerate}
\end{definition}

The value of the relaxed formulation is defined by
\begin{equation}\nonumber
V_R = \sup_{\overline{\mathbb{R}} \in\cR} \mathbb{E}^{\overline{\mathbb{R}}} \left[ \Gamma \right].
\end{equation}

Obviously, we have
\begin{equation}\label{equalSWR}
\cR^{sd}_{S} \subseteq \cR^s_S \subseteq \cR_{S} \subseteq \cR_W \subseteq \cR, \quad\quad \text{and} \quad\quad V^{sd}_{S} \leq V^s_S \leq V_{S} \leq V_W \leq V_R.
\end{equation}

\begin{remark}
	The relaxed formulation in Definition~\ref{defrf} is stated slightly differently with the one in~\cite[Definition 3.4]{MR957652}, but they are mathematically equivalent.
	We just add $L$ as an additional canonical process.
	In particular, the reward function $\Gamma$~\eqref{Gamma} is a continuous function on $\overline{\Omega}$ and then it is easier to prove the convergence result.
\end{remark}

\subsection{Approximating weak control rules by strong control rules}

In this subsection, we prove the equivalence between the weak problem and the strong problem. We first provide a technical lemma. Let
\begin{equation}\label{gammastar}
\gamma^* = \left( \Omega^*, \mathscr{F}^*, \mathbb{Q}^*, \mathbb{F}^*, \mathbb{G}^* , X^*, L^*, B^*, Y^*, u^*= \left\{ u^*_t \right\}_{0 \leq t \leq T} \right)
\end{equation}
be a weak control with a piecewise constant control process over a deterministic time grid $0 = \tau_0 < \tau_1 < \cdots < \tau_N =T$, so that $u^*_t = \mathbf{u}^*_i$ for $t\in[\tau_i, \tau_{i+1})$, where $\mathbf{u}^*_i$ is a $\mathscr{G}_{\tau_i}^*$-measurable random variable.

Further, let us enlarge the space $\Omega^*$ to $\widetilde{\Omega}^* := \Omega^* \times [0,1]^{N}$, on which we obtain an independent sequence of i.i.d. random variables $\left\{ Z_k \right\}_{0 \leq k \leq N-1}$ of uniform distribution on $[0,1]$. Denote the enlarged probability space by $\left( \widetilde{\Omega}^*, \widetilde{\mathscr{F}}^*, \widetilde{\mathbb{Q}}^* \right)$.

\begin{lemma}\label{iidZklemma}
There are measurable functions 
$$
\left\{ \Psi_i : C([0,\tau_i];\mathbb{R}^d) \times [0,1]^{i+1} \longrightarrow U \right\}_{0\leq i \leq N-1}
$$ 
such that
\be\nonumber
\widetilde{\mathbb{Q}}^* \circ \left( B^*, Y^*, \left\{ \Psi_i\left( Y^*_{[0,\tau_i]}, Z_0, \cdots , Z_i \right)\right\}_{0 \leq i \leq N-1} \right)^{-1} = {\mathbb{Q}}^* \circ \left( B^*, Y^*, \left\{ \mathbf{u}^*_i \right\}_{0 \leq i \leq N-1} \right)^{-1}.
\ee
\end{lemma}

\begin{proof}
	First, there are measurable functions (see e.g. \cite[Lemma 4.11]{TanII})
	$$
	\left\{ \Psi_i : C([0,\tau_i];\mathbb{R}^d) \times [0,1]^{i+1} \longrightarrow U \right\}_{0\leq i \leq N-1}
	$$ 
	such that
\begin{equation}\nonumber
\widetilde{\mathbb{Q}}^* \circ \left( Y^*, \left\{ \Psi_i\left( Y^*_{[0,\tau_i]}, Z_0, \cdots , Z_i \right)\right\}_{0 \leq i \leq N-1} \right)^{-1} = 
{\mathbb{Q}}^* \circ \left( Y^*, \left\{ \mathbf{u}^*_i \right\}_{0 \leq i \leq N-1} \right)^{-1}.
\end{equation}
Since $B^*$ is independent with $\mathscr{G}^*_T$, and $\left( Y^*, \left\{ \mathbf{u}^*_i \right\}_{0 \leq i \leq N-1} \right)$ is adapted to $\mathbb{G}^*$, we know that the Brownian motion $B^*$ is independent with $\left( Y^*, \left\{ \mathbf{u}^*_i , Z_i \right\}_{0 \leq i \leq N-1} \right)$. It follows that
\be\nonumber
&&\displaystyle \widetilde{\mathbb{Q}}^* \circ \left( B^*, Y^*, \left\{ \Psi_i\left( Y^*_{[0,\tau_i]}, Z_0, \cdots , Z_i \right)\right\}_{0 \leq i \leq N-1} \right)^{-1} = \left[ \widetilde{\mathbb{Q}}^* \circ \left( B^*\right)^{-1} \right] \otimes \left[ \widetilde{\mathbb{Q}}^* \circ \left( Y^*, \left\{ \Psi_i\left( Y^*_{[0,\tau_i]}, Z_0, \cdots , Z_i \right)\right\}_{0 \leq i \leq N-1} \right)^{-1} \right] \\\nonumber
&=&\displaystyle \left[ {\mathbb{Q}}^* \circ \left( B^*\right)^{-1} \right] \otimes \left[ {\mathbb{Q}}^* \circ \left( Y^*, \left\{ \mathbf{u}^*_i \right\}_{0 \leq i \leq N-1} \right)^{-1} \right] = \left[ {\mathbb{Q}}^* \circ \left( B^*, Y^*, \left\{ \mathbf{u}^*_i \right\}_{0 \leq i \leq N-1} \right)^{-1} \right].
\ee
This completes the proof.
\end{proof}

Next, we show that the weak problem is equivalent to the strong problem.

\begin{lemma}\label{equalWSs}
With the piecewise constant control, the weak problem and the strong problem have the same value, i.e., $V^s_W=V_S^s$.
\end{lemma}

\begin{proof}
The proof is similar to that of~\cite[Theorem 4.10]{TanII}. For convenience of the reader, we give a full proof here. Let us fix an arbitrary weak control $\gamma^*$ defined in~\eqref{gammastar} with its corresponding weak control rule $\overline{\mathbb{R}}^*$, so that one can construct the functionals $\left\{ \Psi_i \right\}_{0\leq i \leq N-1}$ as in Lemma~\ref{iidZklemma}. Following the notations therein, in the probability space $\left( \widetilde{\Omega}^*, \widetilde{\mathscr{F}}^*, \widetilde{\mathbb{Q}}^* \right)$, let us define a control $\widetilde{u}^*_t = \Psi_i \left( Y^*_{[0,\tau_i]}, Z_0, \cdots, Z_i \right)$ for $t\in[\tau_i,\tau_{i+1})$, and processes $\left( \widetilde{X}^*, \widetilde{L}^* \right)$ by $\widetilde{\mathbb{Q}}^*$-a.s.,
\be
		 \left\{
		\ba{rcl}
		&&\displaystyle \widetilde{X}^*_t = {X}^*_0 + \int_0^t b \left(s,\widetilde{X}^*_s,\widetilde{u}^*_s \right) ds + \int_0^t \sigma \left(s,\widetilde{X}^*_s, \widetilde{u}^*_s \right) dB_s^*, \\[0.3cm]
		&&\displaystyle \widetilde{L}^*_t = 1 + \int_0^t p \left(s,\widetilde{X}^*_s \right)  \widetilde{L}^*_s \,dY^*_s.\ea
		\right.\label{strongweakequal}
	\ee
Note that the law $\widetilde{\mathbb{Q}}^* \circ \left( B^*, Y^*, \widetilde{u}^* \right)^{-1} = {\mathbb{Q}}^* \circ \left( B^*, Y^*, u^* \right)^{-1}$, then 
$$
\widetilde{\mathbb{R}}^* := \widetilde{\mathbb{Q}}^* \circ \left( \widetilde{X}^*, \widetilde{L}^*, Y^*, q^{\widetilde{u}^*} \right)^{-1} = {\mathbb{Q}}^* \circ \left( X^*, L^*, Y^*, q^{u^*} \right)^{-1} = \overline{\mathbb{R}}^*.
$$

Let $\left\{ \widetilde{\mathbb{Q}}^*_z \right\}_{z\in [0,1]^N}$ be a family of regular conditional distribution probability of $\widetilde{\mathbb{Q}}^*$ with respect to the $\sigma$-field generated by $\left\{ Z_i \right\}_{0\leq i \leq N-1}$. Then there is a $\widetilde{\mathbb{Q}}^*$-null set $\mathcal{N} \subset [0,1]^N$ such that for each $z\in [0,1]^N \backslash \mathcal{N}$, under $\widetilde{\mathbb{Q}}^*_z$, $(B^*,Y^*)$ is still a Brownian motion and~\eqref{strongweakequal} holds true. Notice that $\widetilde{u}^*$ is adapted to the (augmented) Brownian filtration generated by $Y^*$ under $\widetilde{\mathbb{Q}}^*_z$, thus $\widetilde{\mathbb{R}}^*_z := \widetilde{\mathbb{Q}}^*_z \circ \left( \widetilde{X}^*, \widetilde{L}^*, Y^*, q^{\widetilde{u}^*} \right)^{-1} \in \mathcal{R}_S^s$. It follows that $\mathbb{E}^{\widetilde{\mathbb{R}}^*_z} \left[ \Gamma \right] \leq V^s_S$ for each $z\in [0,1]^N \backslash \mathcal{N}$. And hence
$$
\mathbb{E}^{\overline{\mathbb{R}}^*} \left[ \Gamma \right] = \mathbb{E}^{\widetilde{\mathbb{R}}^*} \left[ \Gamma \right] = \int_{[0,1]^N} \mathbb{E}^{\widetilde{\mathbb{R}}^*_z} \left[ \Gamma \right] dz = \int_{[0,1]^N \backslash \mathcal{N}} \mathbb{E}^{\widetilde{\mathbb{R}}^*_z} \left[ \Gamma \right] dz \leq \int_{[0,1]^N \backslash \mathcal{N}} V^s_S dz = V^s_S.
$$
Thus we get $V^s_W \leq V^s_S$. This completes the proof.
\end{proof}

\subsection{Approximating relaxed control rules by strong/weak control rules}

	We simply recall the following approximation/equivalence results from~\cite{MR957652}.

	\begin{lemma} \label{V0sd}
		$\mathrm{(i)}$
		The set $V^0$ of atomic measures is dense in $V$. More precisely, there exists a sequence $\left\{ \psi_k \right\}_{k\geq 1}$ of measurable maps from $V$ into $V^0$, adapted (i.e., $\psi_k^{-1}(\mathscr{V}_t) \subseteq \mathscr{V}_t$) such that $\psi_k(q)$ converges weakly to $q$ for any $q\in V$. Moreover, we can choose $\psi_k(q)$ such that $\psi_k(q)$ are step (with respect to the time) measures.
	
		\vspace{0.5em}
		
		\noindent $\mathrm{(ii)}$ 
		Let Assumptions $(\mathbf{A1})$ - $(\mathbf{A2})$, and Assumption \ref{definitionYhi} hold true. The weak problem with piecewise constant control shares the same value with the relaxed problem, i.e., $V_R=V_W^s$.
	\end{lemma}

	\begin{proposition}\label{equalRSsd}
		Let Assumptions $(\mathbf{A1})$ - $(\mathbf{A2})$, and Assumption \ref{definitionYhi} hold true, tt holds that $V_R=V^{sd}_{S}$.
	\end{proposition}

	\begin{proof}
		According to Lemma~\ref{equalWSs} and Lemma~\ref{V0sd}, we have $V_S^s=V_R$. Since we can approximate a measurable function $w_i^{s,\kappa}(Y)$ on $C([0,i\kappa];\mathbb{R}^k)$ by 
		$$
		{w}_i^{sd,\kappa}\left(Y_{r_j^i},j\leq J_i,r_j^i \leq i\kappa \right),
		$$
		 where ${w}_i^{sd,\kappa}$ is a uniformly Lipschitz continuous and bounded function on $\mathbb{R}^{k\times J_i}$ with $J_i\in\mathbb{N}_+$, it follows that $
V^s_S=V^{sd}_{S}$. This completes the proof.
	\end{proof}

\section{Proof of the convergence result}\label{sectionproofconvergence}

	In this section, we complete the proof of the main convergence result in Theorem~\ref{ConvergenceTheorem}, which is based on the compactification technique (see also \cite{MR878312,MR957652,Pfeiffer,MR3226164}).
	As preparation, we will first provide some convergence results for the approximating processes.

	Recall that the canonical space $\overline{\Omega}$ is defined by $\overline{\Omega}:=\mathcal{D}^d \times \mathcal{D}^1 \times \mathcal{D}^k \times V$
	with canonical filtration $(\overline{\mathscr{F}}_t)_{t \in [0,T]}$ and canonical process $(X,L,Y,q)$. Throughout the section, we let Assumptions $(\mathbf{A1})$ - $(\mathbf{A2})$, and Assumption \ref{definitionYhi} hold true.

\subsection{Discrete-time control as control rules and its tightness}
\label{preRh}

	We will first redefine a discrete-time control as a probability measure on the canonical space $\left( \overline{\Omega}, \overline{\mathscr{F}} \right)$, 
	so that we can apply the weak convergence technique to derive the convergence of the scheme. 
	Given a discrete-time control  $\left\{ {Y}^h_i, {X}^{h,u^h}_i, {L}^{h,u^h}_i, u^h_i \right\}_{0\leq i \leq n}$,
	we define the continuous time process $\left( \widehat{Y}^h, \widehat{X}^{h,u^h}, \widehat{L}^{h,u^h}, u^h \right)$ by 
	\begin{equation}\nonumber
		\left( \widehat{Y}^h_t, \widehat{X}^{h,u^h}_t, \widehat{L}^{h,u^h}_t, u^h_t \right) 
		=
		\left(Y^h_i, X^{h, u^h}_i, L^{h, u^h}_i, u_i^h \right),
		\quad t\in[t_i,t_{i+1}), \quad i=0, \cdots, n-1.
	\end{equation}
	Notice that $u^h_i$ is $\sigma(Y^h_0, \cdots, Y^h_i)-$measurable, i.e.  $u^h_i:= v_i^h \left(Y^h_0,Y^h_1, \cdots,Y^h_i\right)$.

	We next provide some properties on process $(\widehat{X}^{h, u^h}, \widehat{L}^{h, u^h})$.

	\begin{lemma}\label{barXsupt}
		$\mathrm{(i)}$ There is a constant $C > 0$ independent of $h$, such that
		\begin{equation}\label{hatq}
			\sup_{n\in\mathbb{N}_+} \sup_{0\leq i \leq n-1} \mathbb{E}^{\mathbb{Q}^h} \left[ \left| \widehat{X}_{(i+1)h}^{h,{u^h}} - \widehat{X}_{ih}^{h,{u^h}} \right|^3 
			+
			\left| \widehat{L}_{(i+1)h}^{h,{u^h}} - \widehat{L}_{ih}^{h,{u^h}} \right|^3 \right] 
			\leq
			C h^\frac{3}{2}.
		\end{equation}

		\noindent $\mathrm{(ii)}$
		For any $c>0$, there is a constant $C(c)> 0$ independent of $h$, such that 
		\begin{equation}\nonumber
			\sup_{h>0} \sup_{u^h\in\mathcal{U}^h} \mathbb{E}^{\mathbb{Q}^h} \left[ \sup_{0 \leq t \leq T} \exp\left( c \left| \widehat{X}^{h,{u^h}}_t \right| \right) \right] \leq C(c).
		\end{equation}
		
		\noindent $\mathrm{(iii)}$ There is a constant $C > 0$ independent of $h$, such that
		\begin{equation}\nonumber
			\sup_{h>0} \sup_{u^h\in\mathcal{U}^h} \mathbb{E}^{\mathbb{Q}^h} \left[ \sup\limits_{0\leq s \leq T} \left| \widehat{L}^{h,u^h}_s \right|^{3} \right] \leq C.
		\end{equation}
		
	\end{lemma}

	\begin{proof}
	$\mathrm{(i)}$
	The inequality in \eqref{hatq} is a direct consequence of~\eqref{propertyetahuh} and~\eqref{propertyxihuh}.

	\vspace{0.5em}

	\noindent $\mathrm{(ii)}$
	For the exponential integrability of $\widehat{X}^{h,u^h}$, 
	we notice that the discrete process $\left\{\left| X_k^{h,u^h} - \sum\limits_{i=0}^{k-1} b \left(t_i,X_i^{h,{u^h}},u_i^h \right) h \right| \right\}_{0 \leq k \leq n}$ is a submartingale.
	Hence by Jensen's inequality, for every $c>0$, the discrete process 
	$$
		\left\{ \exp\left(c\left| X_k^{h,u^h} - \sum\limits_{i=0}^{k-1} b \left( t_i,X_i^{h,{u^h}},u_i^h \right)h \right|\right) \right\}_{0 \leq k \leq n}
	$$ 
	is also a positive submartingale. 
	In view of~\eqref{xih4} and Doob's inequality, it holds that
	\be\nonumber
	  \mathbb{E}^{\mathbb{Q}^h} \left[ \sup_{0\leq k \leq n} e^{ c \left| X_k^{h,u^h} \right| } \right] &\leq&\displaystyle C \mathbb{E}^{\mathbb{Q}^h} \left[ \sup_{0\leq k \leq n} \exp\left( c \left| X_k^{h,u^h} - \sum\limits_{i=0}^{k-1} b \left( t_i,X_i^{h,{u^h}},u_i^h \right) h \right| \right) \right] \\\nonumber
		&\leq&\displaystyle C \mathbb{E}^{\mathbb{Q}^h} \left[ \exp\left( c \left| X_n^{h,u^h} - \sum\limits_{i=0}^{n-1} b \left( t_i,X_i^{h,{u^h}},u_i^h \right) h \right| \right) \right] \leq C \mathbb{E}^{\mathbb{Q}^h} \left[ \exp\left( c \left| X_n^{h,u^h} \right| \right) \right] \\\nonumber
		&\leq&\displaystyle  C \mathbb{E}^{\mathbb{Q}^h} \left[ \prod_{i=0}^{n-1} \exp\left( c \left| H_h \left( t_i,X_i^{h,{u^h}},u_i^h,U_{i+1}\right) \right| \right) \right] \leq C(1+Ch)^n \leq C,
	\ee
	where the constant $C$ may depend on $c$, but not on $h$ and $u^h$. 

	\vspace{0.5em}
	
	\noindent $\mathrm{(iii)}$ 
	For $i=1,2, \cdots, n$, we have 
	$$
		{L}^{h,u^h}_{i} = 1 + \sum_{m=0}^{i-1} p \left(t_m,{X}_m^{h,{u^h}}\right) {L}^{h,u^h}_m\,\eta^{h}_{m+1}. 
	$$
According to the discrete BDG inequality, we get
\be\nonumber
 \mathbb{E}^{\mathbb{Q}^h} \left[ \sup_{0 \leq i \leq l+1} \left|{L}^{h,u^h}_{i}\right|^{3} \right] &\leq&\displaystyle C + C \mathbb{E}^{\mathbb{Q}^h} \left[ \sup_{0 \leq i \leq l}\left| \sum_{m=0}^{i} p \left(t_m,{X}_m^{h,{u^h}}\right) {L}^{h,u^h}_m\,\eta^{h}_{m+1} \right|^3\right]  \leq C + C \mathbb{E}^{\mathbb{Q}^h} \left[ \left( \sum_{m=0}^l \left| p \left(t_m,{X}_m^{h,{u^h}}\right) {L}^{h,u^h}_m\,\eta^{h}_{m+1} \right|^2 \right)^\frac{3}{2} \right]\\\nonumber
&\leq&\displaystyle C + C \mathbb{E}^{\mathbb{Q}^h} \left[ n^\frac{1}{2} \sum_{m=0}^l \left| {L}^{h,u^h}_m\,\eta^{h}_{m+1} \right|^3 \right] \leq C + C h  \sum_{m=0}^l \mathbb{E}^{\mathbb{Q}^h} \left[ \left| {L}^{h,u^h}_m \right|^3 \right] \leq C + C h  \sum_{m=0}^l \mathbb{E}^{\mathbb{Q}^h} \left[ \sup_{0 \leq i \leq m}\left| {L}^{h,u^h}_i \right|^3 \right].
\ee
Then using Gronwall inequality, we have 
$$
\mathbb{E}^{\mathbb{Q}^h} \left[ \sup\limits_{0\leq s \leq T} \left| \widehat{L}^{h,u^h}_s \right|^{3} \right] \leq \mathbb{E}^{\mathbb{Q}^h} \left[ \sup\limits_{0\leq i \leq n} \left|{L}^{h,u^h}_{i}\right|^{3} \right] \leq C,
$$
where the constant $C$ is independent with $h$ and $u^h$. This completes the proof.
\end{proof}

	Next, for every discrete-time control $u^h$, let us define 
	$$
		q^{h,{u^h}}(dt,da):=\delta_{u_t^h}(da)dt.
	$$
	and
	$$
		\cR_h = \left\{ \overline{\mathbb{R}}_{h,{u^h}}: u^h\in\cU_h \right\},
		\quad\mbox{with}\quad
		\overline{\mathbb{R}}_{h,{u^h}} = \mathbb{Q}^h \circ ( \widehat{X}^{h,{u^h}}, \widehat{L}^{h,{u^h}}, {\widehat{Y}^h},q^{h,{u^h}})^{-1}.
	$$
	Then it is clear that
	\begin{equation}\label{defbarVhcanonical}
		V_h = \sup_{\overline{\mathbb{R}}_{h,{u^h}}\in\cR_h}\mathbb{E}^{\overline{\mathbb{R}}_{h,{u^h}}} \left[ \Gamma \right].
	\end{equation}

To study the convergence of approximating processes $ \left( \widehat{X}^{h,{u^h}}, \widehat{L}^{h,{u^h}}, {\widehat{Y}^h},q^{h,{u^h}} \right) $, we first need to show the tightness of the laws of these processes $\left\{ \overline{\mathbb{R}}_{h,{u^h}}: u^h\in\cU_h \right\}_{h>0}$ on the enlarged canonical space $(\overline{\Omega},\overline{\mathscr{F}})$.

	\begin{lemma}\label{precompactRhLemma}
		The collection of measures 
		$$
		\left\{ \overline{\mathbb{R}}_{h,{u^h}} = \mathbb{Q}^h \circ \left( \widehat{X}^{h,{u^h}}, \widehat{L}^{h,{u^h}}, {\widehat{Y}^h},q^{h,{u^h}} \right)^{-1} : u^h\in\cU_h \right\}_{h>0}
		$$ 
		on $\overline{\Omega}$ is tight.
	\end{lemma} 

\begin{proof}

By~\eqref{xih3} and~\eqref{hatq}, it follows the tightness of $\left\{ \overline{\mathbb{R}}_{h,{u^h}}\big|_{\mathcal{D}^d}: u^h\in\cU_h \right\}_{h>0}$, 
	see e.g.~\cite[Proposition VI.3.26]{jacod2013limit} together with \cite[Theorem 2.4.10, Problem 2.4.11]{Karatzas} or \cite[Theorem 1.4.11]{MR532498}.
	Similarly, by~\eqref{propertyetahuh},~\eqref{hatq} and Lemma~\ref{barXsupt}, we know that the set of measures $\left\{ \overline{\mathbb{R}}_{h,{u^h}}\big|_{\mathcal{D}^1 \times \mathcal{D}^k}: u^h\in\cU_h \right\}_{h>0}$ on $\mathcal{D}^1 \times \mathcal{D}^k$ is tight.

Then, we prove that the set of measures $\left\{ \overline{\mathbb{R}}_{h,{u^h}}\big|_{V}: u^h\in\cU_h \right\}_{h>0}$ on $V$ is tight. Note that $A$ is a compact Polish space. According to Prokhorov's theorem, we know that $\mathbf{M}([0,T]\times A)$ is compact under the weak convergence topology. Thus $V$ is also compact as a closed subset of $\mathbf{M}([0,T]\times A)$. According to Prokhorov's theorem, the class of probability measures $\left\{ \overline{\mathbb{R}}_{h,{u^h}}\big|_{V} \right\}$ on $V$ is tight.  This completes the proof.
\end{proof}

\subsection{A completion functional space on $\overline{\Omega}$}

To prove the convergence results, we follow~\cite{MR3226164} to introduce a space $\overline{L}^1_*$ of random variables on $\overline{\Omega}$, and prove that the reward function $\Gamma$ defined in~\eqref{Gamma} belongs to this space. Let $\overline{\cR}^* := \left( \mathop{\bigcup}\limits_{h>0} \cR_h \right) \cup \cR$,
and define a norm $\left\| \,\cdot\, \right\|_*$ for random variables on $\overline{\Omega}$ by $\left\| \xi \right\|_* = \sup\limits_{\overline{\mathbb{P}} \in \overline{\cR}^* } \mathbb{E}^{\overline{\mathbb{P}}} \left[ \,\left| \xi \right|\, \right]$. Denote by $\overline{L}^1_*$ the completion space of $C_b \left(\overline{\Omega}\right)$ under the norm $\left\| \,\cdot\, \right\|_*$.

Firstly, we derive a convergence result for random variables in $\overline{L}^1_*$.

\begin{lemma}\label{L1*convergence}
Suppose that $\xi \in \overline{L}^1_*$, and $\left\{ \overline{\mathbb{P}}_n \right\}_{n\geq 1} \subseteq \overline{\cR}^*$ such that $ \overline{\mathbb{P}}_n  \longrightarrow \overline{\mathbb{P}}_\infty \in \overline{\cR}^*$ weakly. Then $\mathbb{E}^{\overline{\mathbb{P}}_n} \left[ \xi \right] \longrightarrow \mathbb{E}^{\overline{\mathbb{P}}_\infty} \left[ \xi \right]$.
\end{lemma}

\begin{proof}
For every $\varepsilon > 0$, there is $\xi_\varepsilon \in C_b \left(\overline{\Omega}\right)$ such that 
$$
\left\| \xi - \xi_\varepsilon \right\|_* = \sup\limits_{\overline{\mathbb{P}} \in \overline{\cR}^* } \mathbb{E}^{\overline{\mathbb{P}}} \left[ \,\left| \xi - \xi_\varepsilon \right|\, \right] \leq \varepsilon.
$$ 
It follows that
\be\nonumber
\varlimsup_{n\rightarrow \infty} \left| \mathbb{E}^{\overline{ \mathbb{P}}_n} \left[ \xi \right] - \mathbb{E}^{\overline{\mathbb{P}}_\infty} \left[ \xi \right] \right| \leq \sup_{n\geq 1} \mathbb{E}^{\overline{\mathbb{P}}_n} \left[ \left| \xi - \xi_\varepsilon \right| \right] + \varlimsup_{n\rightarrow \infty} \left| \mathbb{E}^{\overline{\mathbb{P}}_n} \left[ \xi_\varepsilon \right] - \mathbb{E}^{\overline{\mathbb{P}}_\infty} \left[ \xi_\varepsilon \right] \right| + \mathbb{E}^{\overline{\mathbb{P}}_\infty} \left[ \left| \xi_\varepsilon - \xi \right| \right] \leq 2\varepsilon.
\ee
This completes the proof by the arbitrariness of $\varepsilon$.
\end{proof}

\begin{lemma}\label{convergencePhi}
Let $\Phi\in C \left(\overline{\Omega}\right)$. If $\sup\limits_{\overline{\mathbb{P}} \in \overline{\cR}^* } \mathbb{E}^{\overline{\mathbb{P}}} \left[ \left| \Phi \right|^2 \right] <\infty$, then $\Phi\in\overline{L}^1_*$. 
\end{lemma}

\begin{proof}
Define $\Phi_N:= (-N) \vee (\Phi \wedge N)$. Then $\Phi_N \in C_b \left( \overline{\Omega} \right)$. According to Cauchy-Schwarz inequality, we have
\be\nonumber
\displaystyle \left\| \Phi - \Phi_N \right\|_* &=&\displaystyle \sup_{\overline{\mathbb{P}} \in \overline{\cR}^*} \mathbb{E}^{\overline{\mathbb{P}}} \left[ \left| \Phi - \Phi_N \right| \right] = \sup_{\overline{\mathbb{P}} \in \overline{\cR}^*} \mathbb{E}^{\overline{\mathbb{P}}} \left[ \left| \Phi - \Phi_N \right| \cdot 1_{\{|\Phi|>N\}} \right] \leq \sup_{ \overline{\mathbb{P}} \in \overline{\cR}^* } \left\{ \left( \mathbb{E}^{\overline{\mathbb{P}}} \left[ \left| \Phi - \Phi_N \right|^2 \right] \right)^\frac{1}{2} \left( \overline{\mathbb{P}} \left[ \left| \Phi \right|>N \right] \right)^\frac{1}{2} \right\}\\\nonumber
&\leq& \displaystyle \sup_{ \overline{\mathbb{P}} \in \overline{\cR}^*} \left\{ \left( 4 \mathbb{E}^{\overline{\mathbb{P}}} \left[ \left| \Phi \right|^2 \right] \right)^\frac{1}{2}\right\}  \sup_{ \overline{\mathbb{P}} \in \overline{\cR}^*} \left\{ \left( \frac{\mathbb{E}^{\overline{\mathbb{P}}} \left[ \left| \Phi \right|^2 \right]}{N^2} \right)^\frac{1}{2} \right\} \leq \frac{C}{N}.
\ee
Thus we have $\left\| \Phi - \Phi_N \right\|_* \longrightarrow 0$ as $N$ tends to infinity, and hence $\Phi\in\overline{L}^1_*$.
\end{proof}

Next, we show that the reward functional $\Gamma$ on the enlarged canonical space $\overline{\Omega}$ belongs to $\overline{L}^1_*$.

\begin{lemma} \label{GammainL1*}
The random variable $\Gamma$ defined by~\eqref{Gamma} lies in $\overline{L}^1_*$.
\end{lemma}

\begin{proof}
Under the assumption ($\mathbf{A2}$), we have
\begin{equation}\label{boundKG}
\left|\int_0^T K(s,X_s,u_s)\,ds\right| + \left| G(X_T)\right| \leq C \left[ 1+ \sup_{0\leq t \leq T} \exp \left( C \left| X_t \right| \right) \right].
\end{equation}
In view of Lemma~\ref{barXsupt}, we have
\begin{equation}\label{supcRh}
\sup_{h>0}\,\, \sup_{\overline{\mathbb{R}}_{h,u^h} \in \cR_h } \mathbb{E}^{\overline{\mathbb{R}}_{h,u^h}} \left[ \left| \Gamma \right|^2 \right] <\infty.
\end{equation}

Next, we show that $\sup\limits_{\overline{\mathbb{R}} \in \cR } \mathbb{E}^{\overline{\mathbb{R}}} \left[ \left| \Gamma \right|^2 \right] <\infty$. Fix $\overline{\mathbb{R}} \in \cR$. In view of~\eqref{eq:dynamicXL}, since the coefficients $\overline{b}$ and $\overline{\Sigma}^\frac{1}{2}$ are uniformly bounded, $X$ is a continuous semi-martingale whose finite variation part and quadratic variation part are both bounded by a constant $M_T$ under $\overline{\mathbb{R}}$. When $d=1$, by Dambis-Dubins-Schwarz's time change theorem, we get
$$
\mathbb{E}^{\overline{\mathbb{R}}} \left[ \sup_{0\leq t\leq T} \exp \left( C | X_t | \right) \right] \leq e^{C M_T} \mathbb{E}^{\overline{\mathbb{R}}} \left[ \sup_{0\leq t\leq M_T} \exp( C | {B}^\#_t |) \right]\leq C,
$$
where ${B}^\#$ is a standard 1-d Brownian motion. When $d>1$, it is enough to remark that for $X= \left(X^1, \cdots, X^d \right)$, we have
\be\nonumber
 \sup_{0\leq t\leq T} \exp\left( C \left| X_t \right|\right) = \exp\left( C\sup_{0\leq t\leq T} \left| X_t \right| \right) \leq C \exp\left\{ C\left( \sup_{0\leq t\leq T} \left| X^1_t \right| + \cdots + \sup_{0\leq t\leq T} \left| X^d_t \right| \right) \right\}.
\ee
Thus we get
\begin{equation}\label{supt[x]}
\sup_{\overline{\mathbb{R}} \in \cR} \mathbb{E}^{\overline{\mathbb{R}}} \left[ \sup_{0\leq t\leq T} \exp\left( c \left| X_t \right|\right) \right] \leq C.
\end{equation}
Since $L$ is an exponential martingale with uniformly bounded coefficient $p$, and according to~\eqref{boundKG} and~\eqref{supt[x]}, we get
\begin{equation}\label{supcRtilde}
\sup_{\overline{\mathbb{R}} \in \cR } \mathbb{E}^{\overline{\mathbb{R}}} \left[ \left| \Gamma \right|^2 \right] <\infty.
\end{equation}

Finally, in view of~\eqref{supcRh} and~\eqref{supcRtilde}, we have $\sup\limits_{\overline{\mathbb{P}} \in \overline{\cR}^* } \mathbb{E}^{\overline{ \mathbb{P}}} \left[ \left| \Gamma \right|^2 \right] <\infty$. By Lemma~\ref{convergencePhi}, we get $\Gamma\in\overline{L}^1_*$.
\end{proof}

\subsection{Approximating martingale property}

Since our numerical schemes for the discrete processes~\eqref{defYhi},~\eqref{defXhuhi} and~\eqref{Lhuhi} are quite general, with our fairly general assumptions on the coefficients, we can only prove the weak convergence results, i.e., the tightness of distribution of the approximating processes on the enlarged canonical space. In this case, the martingale property is expected to be inherited as the time mesh size $h$ tends to $zero$, so that we can justify that the limit distribution is the desired distribution of the optimal state process. In the following lemma, we study the converging martingale property for the approximating signal processes $\widehat{X}^{h,{u^h}}$ and for the approximating Radon-Nikodym derivative $\widehat{L}^{h,{u^h}}$.

\begin{lemma}\label{martingaleproperty}
Suppose that a sequence $\left\{ \overline{\mathbb{R}}_{h,{u^h}} \right\} \subseteq \left\{ \cR_h \right\}_{h>0}$ converges weakly to $\overline{\mathbb{Q}}$ as $h\rightarrow 0$. Then the process $\left\{C_t^{f}(X,q)\right\}_{0\leq t\leq T}$ defined by~\eqref{CtfB[x][z]} is a $\left\{\overline{\mathscr{G}}_T \vee \overline{\mathscr{F}}_t \right\}_{0\leq t\leq T}$ -martingale under $\overline{\mathbb{Q}}$.

\end{lemma}

\begin{proof}

Since the random variable $U^h_{i+1}$ is independent with the $\sigma$-algebra $\sigma \left( {Y}^h \right):=\sigma\left(Y^h_i:i=0,1,\cdots,n\right)$, the integral properties for $H_h \left( t_i , X_i^{h,{u^h}}, u^h_i, U^h_{i+1} \right)$ in~\eqref{propertyxihuh} hold still for $\mathbb{E}^Y_i\left[ \,\cdot\, \right]: = \mathbb{E}^{\mathbb{Q}^h} \left[ \,\,\cdot\,\, \big|\mathscr{F}_{i}^h \vee \sigma\left({Y}^h\right)\right]$.

Note that $\widehat{X}^{h,{u^h}}_{t_{i+1}} - \widehat{X}^{h,{u^h}}_{t_i} = H_h \left( t_i, X_i^{h,{u^h}}, u^h_i, U^h_{i+1} \right)$. By Taylor expansion, we know that for each $f\in C_b^{\infty}([0,T]\times \mathbb{R}^d)$,
\begin{equation}\label{boundedpf}
\mathbb{E}^Y_i\left[f\left( t_{i+1}, \widehat{X}^{h,{u^h}}_{t_{i+1}} \right) - f\left( t_i, \widehat{X}^{h,{u^h}}_{t_i} \right)- \cL^{t_i,\widehat{X}^{h,{u^h}}_{t_i},\,u^h_{t_i}}f\left( t_i, \widehat{X}^{h,{u^h}}_{t_i} \right) h\right] = \varepsilon_i^h,
\end{equation}
where $| \varepsilon_i^h | \leq C h^\frac{3}{2} $. Thus, we have
$$
\mathbb{E}^Y_i\left[f \left( t_{i+1}, \widehat{X}^{h,{u^h}}_{t_{i+1}} \right) - f \left( t_i, \widehat{X}^{h,{u^h}}_{t_i} \right) - \int_{t_i}^{t_{i+1}} \cL^{r,\widehat{X}^{h,{u^h}}_r,\,u^h_r}f(r,\widehat{X}^{h,{u^h}}_r)\,dr\right] = \theta^h_i,
$$
where 
$$
\theta^h_i := \mathbb{E}^Y_i \left[\int_{t_i}^{t_{i+1}} \left\{ \cL^{t_i,\widehat{X}^{h,{u^h}}_{t_i},u^h_{t_i}}f \left( t_i ,\widehat{X}^{h,{u^h}}_{t_i} \right) - \cL^{r,\widehat{X}^{h,{u^h}}_r,\,u^h_r}f\left(r,\widehat{X}^{h,{u^h}}_r\right) \right\} dr\right] + \varepsilon^h_i.
$$
By~\eqref{hatq}, we get 
\begin{equation}\nonumber
\mathbb{E}^{\mathbb{Q}^h} \left[ \left| \theta^h_i \right| \right] \leq Ch \left[ \rho(h) + h^\frac{1}{2} \right].
\end{equation} 
It holds that
\be\nonumber
&&\displaystyle \left| \mathbb{E}^Y_s \left[f \left(t,\widehat{X}^{h,{u^h}}_{t} \right)-f \left(s,\widehat{X}^{h,{u^h}}_{s}\right)-\int_{s}^{t} \cL^{r,\widehat{X}^{h,{u^h}}_r,\,u^h_r}f \left(r,\widehat{X}^{h,{u^h}}_r\right)dr\right] \right| \\\nonumber
&\leq&\displaystyle \left| \mathbb{E}^Y_s\left[f \left(t,\widehat{X}^{h,{u^h}}_{t}\right)-f \left(n^h_th,\widehat{X}^{h,{u^h}}_{n^h_th}\right)-\int_{n^h_th}^{t} \cL^{r,\widehat{X}^{h,{u^h}}_r,\,u^h_r}f \left(r,\widehat{X}^{h,{u^h}}_r\right) dr \right] \right|
\\\nonumber
&&\displaystyle +\mathbb{E}^Y_s \Bigg\{  \sum_{i=n^h_s+1}^{n^h_t-1} \bigg| \mathbb{E}^Y_i  \bigg[f \left( t_{i+1},\widehat{X}^{h,{u^h}}_{t_{i+1}}\right) - f \left( t_i,\widehat{X}^{h,{u^h}}_{t_i}\right) -\int_{t_i}^{t_{i+1}} \cL^{r,\widehat{X}^{h,{u^h}}_r,\,u^h_r}f \left(r,\widehat{X}^{h,{u^h}}_r \right)dr\bigg] \bigg| \Bigg\} \\\nonumber
&&\displaystyle + \bigg| \mathbb{E}^Y_s\bigg[f \left((n^h_s+1)h,\widehat{X}^{h,{u^h}}_{(n^h_s+1)h}\right) - f \left(s,\widehat{X}^{h,{u^h}}_{s}\right) - \int_{s}^{(n^h_s+1)h} \cL^{r,\widehat{X}^{h,{u^h}}_r,\,u^h_r} f\left(r,\widehat{X}^{h,{u^h}}_r\right) dr\bigg] \bigg| \\\label{martingaleconvergence}
\quad\quad&\leq&\displaystyle \Delta_h,
\ee
where $n_s^h := t_i$ for $s\in[t_i,t_{i+1})$, $\mathbb{E}^Y_s\left[ \,\cdot\, \right]: = \mathbb{E}^Y_{n^h_s+1}\left[ \,\cdot\, \right]$ and $\mathbb{E}^{\mathbb{Q}^h} \left[ \left| \Delta_h \right| \right] \leq C \left[ \rho(h) + h^\frac{1}{2} \right]$.

To prove that the process $C_t^{f}(X,q)$ is a $\overline{\mathscr{G}}_T \vee \overline{\mathscr{F}}_t$ - martingale under $\overline{\mathbb{Q}}$, it is enough to show that for any $f\in C_b^\infty([0,T]\times\mathbb{R}^d$),
\be\label{hmartingale}
 \mathbb{E}^{\overline{\mathbb{R}}_{h,{u^h}}} \bigg[ \Phi\left( X_{s_i}, L_{s_i}, Y_{r_j}, q_{r_j}(\phi_m); \,\,i\leq I, j\leq J, m\leq M \right) \times \left\{ C_t^{f}\left(X,q\right) - C_s^{f}\left(X,q\right) \right\} \bigg]
\ee
tends to $zero$ as $h\rightarrow 0$, for arbitrary $I,J,M\in\mathbb{N}_+$, $0\leq s_i \leq s <t \leq T$, $0\leq r_j \leq T$, $\phi_m\in C_b([0,T]\times a)$, $\Phi \in C_b\left( \mathbb{R}^{d\times I} \times \mathbb{R}^{I} \times \mathbb{R}^{k\times J} \times \mathbb{R}^{M \times J} \right)$. Since $\overline{\mathbb{R}}_{h,{u^h}} = \mathbb{Q}^h \circ (\widehat{X}^{h,{u^h}}, \widehat{L}^{h,{u^h}}, {\widehat{Y}^h}, q^{h,{u^h}})^{-1}$, we know that~\eqref{hmartingale} is equal to
\be\nonumber
&&\displaystyle \mathbb{E}^{\mathbb{Q}^h} \Bigg[ \Phi\left( \widehat{X}^{h,{u^h}}_{s_i}  \,,\, \widehat{L}^{h,{u^h}}_{s_i}  \,,\, \widehat{Y}^h_{r_j} \,,\, \int_0^{r_j} \phi_m \left( r, u^h_r\right) dr \,;\,\, i\leq I\,,\, j\leq J\,,\, m\leq M \right) \\\nonumber
&&\displaystyle \quad\quad\quad\quad \times \left\{ f\left(t,\widehat{X}^{h,{u^h}}_{t}\right)-f \left(s,\widehat{X}^{h,{u^h}}_{s}\right)-\int_{s}^{t} \cL^{r,\widehat{X}^{h,{u^h}}_r,\,u^h_r}f \left(r,\widehat{X}^{h,{u^h}}_r \right)dr \right\} \Bigg] .
\ee
Since $s_i \leq s \leq \left(n^h_s +1 \right)h$ for $i=1, \cdots,I$, the above term equals to
\be\nonumber
&&\displaystyle \mathbb{E}^{\mathbb{Q}^h} \Bigg[ \Phi\left( \widehat{X}^{h,{u^h}}_{s_i} \,,\, \widehat{L}^{h,{u^h}}_{s_i} \,,\, \widehat{Y}^h_{r_j} \,,\, \int_0^{r_j} \phi_m \left( r, u^h_r\right) dr \,;\,\, i\leq I \,,\, j\leq J \,,\, m\leq M\right) \\\nonumber
&&\displaystyle \quad\quad\quad \times \mathbb{E}^Y_s\left\{ f\left(t,\widehat{X}^{h,{u^h}}_{t}\right)-f \left(s,\widehat{X}^{h,{u^h}}_{s}\right) - \int_{s}^{t} \cL^{r,\widehat{X}^{h,{u^h}}_r,u^h_r}f \left(r,\widehat{X}^{h,{u^h}}_r \right)dr \right\} \Bigg],
\ee
which is bounded by $C_\Phi \mathbb{E}^{\mathbb{Q}^h} \left[ \left| \Delta_h \right| \right] \leq C [ \rho(h) + h^\frac{1}{2} ]$ according to~\eqref{martingaleconvergence}. This completes the proof.
\end{proof}

\begin{lemma}\label{martingalepropertyLY}
Suppose that a sequence $\left\{ \overline{\mathbb{R}}_{h,{u^h}} \right\} \subseteq \left\{ \cR_h \right\}_{h>0}$ converges weakly to $\overline{\mathbb{Q}}$ as $h\rightarrow 0$. Then the process $\left\{ D_t^{g}(X,L,Y) \right\}_{0\leq t\leq T}$ defined by~\eqref{Dtg} is a $\left\{\overline{\mathscr{F}}_t \right\}_{0\leq t\leq T}$ -martingale under $\overline{\mathbb{Q}}$.

\end{lemma}

\begin{proof}

By Taylor expansion, we know that for each $g\in C^\infty_b([0,T]\times \mathbb{R} \times \mathbb{R}^k)$,
\be\label{L2pf}
\mathbb{E}_i\bigg[g\left(t_{i+1}, \widehat{L}^{h,{u^h}}_{t_{i+1}}, \widehat{Y}^{h}_{t_{i+1}} \right) - g\left(t_i, \widehat{L}^{h,{u^h}}_{t_i}, \widehat{Y}^{h}_{t_i} \right) - \cM^{t_i,\widehat{X}^{h,{u^h}}_{t_i},\widehat{L}^{h,{u^h}}_{t_i} } g\left(t_i, \widehat{L}^{h,{u^h}}_{t_i}, \widehat{Y}^{h}_{t_i} \right) h\bigg] = \hat{\varepsilon}_i^h,
\ee
where $\mathbb{E}_i\left[\,\cdot\,\right] = \mathbb{E}^{\mathbb{Q}^h} \left[\,\cdot\,\big| \mathscr{F}_{i}^h \right]$ and $\left| \hat{\varepsilon}_i^h \right| \leq C \left( 1 + \left| \widehat{L}^{h,{u^h}}_{t_i} \right|^3 \right)  h^\frac{3}{2} $. Compared with~\eqref{boundedpf}, the additional term $C \left( 1 + \left| \widehat{L}^{h,{u^h}}_{t_i} \right|^3 \right)$ in $\hat{\varepsilon}_i^h$ comes from the unbounded increment of the approximating Radon-Nikodym derivative process 
$$
\widehat{L}^{h,{u^h}}_{t_{i+1}} - \widehat{L}^{h,{u^h}}_{t_i} = p \left( t_i, X_i^{h,{u^h}} \right) \widehat{L}^{h,{u^h}}_{t_i} \eta^h_{i+1}.
$$ 
However, the uniform integrability of $\widehat{L}^{h,{u^h}}$ in Lemma~\ref{barXsupt} can help us to deal with this issue, which gives that $\mathbb{E}^{\mathbb{Q}^h} \left[ \left| \hat{\varepsilon}_i^h \right| \right] \leq C h^\frac{3}{2}$. Then we have
\be\nonumber
\mathbb{E}_i\bigg[g\left(t_{i+1}, \widehat{L}^{h,{u^h}}_{t_{i+1}}, \widehat{Y}^{h}_{t_{i+1}} \right) - g\left(t_i, \widehat{L}^{h,{u^h}}_{t_i}, \widehat{Y}^{h}_{t_i} \right) - \int_{t_i}^{t_{i+1}} \cM^{r,\widehat{X}^{h,{u^h}}_{r},\widehat{L}^{h,{u^h}}_{r} } g\left(r, \widehat{L}^{h,{u^h}}_{r}, \widehat{Y}^{h}_{r} \right) dr\bigg] = \hat{\theta}_i^h,
\ee
where $\mathbb{E}^{\mathbb{Q}^h} \left[ \left| \hat{\theta}_i^h \right| \right] \leq Ch [ \rho(h) + h^\frac{1}{2} ]$ according to~\eqref{hatq} and Lemma~\ref{barXsupt}. It follows that
\be\label{martingaleconvergenceLY}
\bigg| \mathbb{E}_s \bigg[g\left(t, \widehat{L}^{h,{u^h}}_{t}, \widehat{Y}^{h}_{t} \right) - g\left(s, \widehat{L}^{h,{u^h}}_{s}, \widehat{Y}^{h}_{s} \right) - \int_{s}^{t} \cM^{r,\widehat{X}^{h,{u^h}}_{r},\widehat{L}^{h,{u^h}}_{r} } g\left(r, \widehat{L}^{h,{u^h}}_{r}, \widehat{Y}^{h}_{r} \right) dr \bigg] \bigg| \leq \hat{\Delta}_h,
\ee
in which $\mathbb{E}_s\left[ \,\cdot\, \right]: = \mathbb{E}_{n^h_s+1}\left[ \,\cdot\, \right]$ and $\mathbb{E}^{\mathbb{Q}^h} \left[ \left| \hat{\Delta}_h \right| \right] \leq C [ \rho(h) + h^\frac{1}{2} ]$.

	Notice that there exists a countable set $\mathbb{T}_0 \subset (0,T)$, such that, for arbitrary $I,M\in\mathbb{N}_+$,  $\phi_m\in C_b([0,T]\times a)$, $\hat{\Phi} \in C_b\left( \mathbb{R}^{d\times I} \times \mathbb{R}^{I} \times \mathbb{R}^{k\times I} \times \mathbb{R}^{M \times I} \right)$, and $0\leq s_i \leq s <t \leq T$ with $s, t \in [0,T] \setminus \mathbb{T}_0$, the functional $\Theta(X,L,Y,q) : \overline{\Omega} \longrightarrow \R$ defined as follows is $\overline{\mathbb{Q}}$-a.s. continuous,
\be\nonumber
\Theta\left( X,L,Y,q \right):=\hat{\Phi} \left( X_{s_i}, L_{s_i}, Y_{s_i}, q_{s_i}(\phi_m); i\leq I, m\leq M \right) \left[ D_t^{g}(X,L,Y) - D_s^{g}(X,L,Y) \right].
\ee
One can check that $\left| \Theta \right| \leq C\left( 1 + \sup\limits_{0\leq r \leq T} \left| L_r \right|^2 \right) $. In view of Lemma~\ref{barXsupt}, Lemma~\ref{L1*convergence} and Lemma~\ref{convergencePhi}, we know that $\mathbb{E}^{\overline{\mathbb{R}}^{h,u^h}} \left[ \Theta \right] \longrightarrow \mathbb{E}^{\overline{\mathbb{Q}}} \left[ \Theta \right]$. Thus, to prove that the process $D_t^{g}(X,L,Y)$ is a $\overline{\mathscr{F}}_t$ - martingale under $\overline{\mathbb{Q}}$, it is enough to show that $\mathbb{E}^{\overline{\mathbb{R}}^{h,u^h}} \left[ \Theta \right] \longrightarrow 0$. Since $\overline{\mathbb{R}}_{h,{u^h}} = \mathbb{Q}^h \circ \left(\widehat{X}^{h,{u^h}}, \widehat{L}^{h,{u^h}}, {\widehat{Y}^h}, q^{h,{u^h}} \right)^{-1}$ and $s_i \leq s \leq \left(n^h_s +1 \right)h$ for $i=1,\cdots,I$, we get
\be\nonumber
\mathbb{E}^{\overline{\mathbb{R}}^{h,u^h}} \left[ \Theta \right] &=&\displaystyle \mathbb{E}^{\mathbb{Q}^h} \Bigg[ \hat{\Phi} \left( \widehat{X}^{h,{u^h}}_{s_i}  \,,\, \widehat{L}^{h,{u^h}}_{s_i}  \,,\, \widehat{Y}^h_{s_i} \,,\, \int_0^{s_i} \phi_m \left( r, u^h_r\right) dr; \,\,i\leq I \,,\, m\leq M \right)\\\nonumber 
&&\displaystyle \times \left\{ g\left(t, \widehat{L}^{h,{u^h}}_{t}, \widehat{Y}^{h}_{t} \right) - g\left(s, \widehat{L}^{h,{u^h}}_{s}, \widehat{Y}^{h}_{s} \right)- \int_{s}^{t} \cM^{r,\widehat{X}^{h,{u^h}}_{r},\widehat{L}^{h,{u^h}}_{r} } g\left(r, \widehat{L}^{h,{u^h}}_{r}, \widehat{Y}^{h}_{r} \right) dr \right\} \Bigg] \\\nonumber
&=&\displaystyle \mathbb{E}^{\mathbb{Q}^h} \Bigg[ \hat{\Phi} \left( \widehat{X}^{h,{u^h}}_{s_i}  \,,\, \widehat{L}^{h,{u^h}}_{s_i}  \,,\, \widehat{Y}^h_{s_i} \,,\, \int_0^{s_i} \phi_m \left( r, u^h_r\right) dr; \,\,i\leq I \,,\, m\leq M \right)\\\nonumber 
&&\displaystyle \times \mathbb{E}_s  \left[ g\left(t, \widehat{L}^{h,{u^h}}_{t}, \widehat{Y}^{h}_{t} \right) - g\left(s, \widehat{L}^{h,{u^h}}_{s}, \widehat{Y}^{h}_{s} \right)- \int_{s}^{t} \cM^{r,\widehat{X}^{h,{u^h}}_{r},\widehat{L}^{h,{u^h}}_{r} } g\left(r, \widehat{L}^{h,{u^h}}_{r}, \widehat{Y}^{h}_{r} \right) dr \right] \Bigg],
\ee
which is bounded by $C_{\hat{\Phi}} \,\mathbb{E}^{\mathbb{Q}^h} \left[ \left| \hat{\Delta}_h \right| \right] \leq C [\rho(h) + h^\frac{1}{2}]$ according to~\eqref{martingaleconvergenceLY}. This completes the proof.
\end{proof}

\subsection{Proof of Theorem~\ref{ConvergenceTheorem}}\label{proofmainresultsub}

	We finally provide the proof of our main result.
	Let us first prove some important Lemmas.

\begin{lemma}\label{Rhtight}
Let $\{ \overline{\mathbb{R}}_{h_m,{u^{h_m}}} \}_{m\in\mathbb{N}_+} \subseteq \{ \cR_h \}_{h>0}$ be a sequence of approximating control rules and $\lim_{m\rightarrow \infty} h_m =0$. Then $\{ \overline{\mathbb{R}}_{h_m,{u^{h_m}}} \}_{m\in\mathbb{N}_+}$ is tight, and any cluster point belongs to the set of relaxed control rules $\cR$.
\end{lemma}

\begin{proof}
According to Lemma~\ref{precompactRhLemma}, we know that $\{ \overline{\mathbb{R}}_{h_m,{u^{h_m}}} \}_{m\in\mathbb{N}_+}$ is tight. For each cluster $\overline{\mathbb{Q}}$, by Lemma~\ref{martingaleproperty}, the process $ \left\{ C_t^f(X,q) \right\}_{0\leq t\leq T}$ defined by~\eqref{CtfB[x][z]} is a $\left\{\overline{\mathscr{G}}_T \vee \overline{\mathscr{F}}_t \right\}_{0\leq t\leq T}$ - martingale, and by Lemma~\ref{martingalepropertyLY}, the process $\left\{ D_t^{g}(X,L,Y) \right\}_{0\leq t\leq T}$ defined by~\eqref{Dtg} is a $\left\{\overline{\mathscr{F}}_t \right\}_{0\leq t\leq T}$ -martingale. We have verified that the measure $\overline{\mathbb{Q}}$ satisfies all the conditions of Definition~\ref{defrf}, and thus belongs to $\cR$.
\end{proof}

Next, by the uniqueness of the solution to the martingale problem, the convergence results can help us to show that any feasible control rule in $\cR_S^{sd}$ can be approximated by the discrete problem when $h$ tends to $zero$.

\begin{lemma}\label{constructRs0}
Let $\overline{\mathbb{R}}_S^{sd} \in \cR_S^{sd}$. Then we can construct a sequence of approximating control rules $\left\{ \overline{\mathbb{R}}_{h,{\overline{u}^h}} \right\}_{h>0}$ such that $\overline{\mathbb{R}}_{h,{\overline{u}^h}} \in \cR_h$ and $\overline{\mathbb{R}}_{h,{\overline{u}^h}} \rightarrow \overline{\mathbb{R}}_S^{sd}$ as $h\rightarrow 0$.
\end{lemma}

\begin{proof}
 
Let $\kappa > 0$. We suppose that $\overline{\mathbb{R}}^{sd}_{S} = \mathbb{Q} \circ \left(X^{u^{sd,\kappa}}, L^{u^{sd,\kappa}}, Y,\, q^{u^{sd,\kappa}}\right)^{-1}$, where
$$
u^{sd,\kappa}_t = u^{sd,\kappa}_t (Y) := w_m^{sd,\kappa} \left(Y_{r_j^m} \,;\, j\leq J_m,r_j^m \leq m\kappa \right),\quad\quad t\in[m\kappa,(m+1)\kappa),
$$
with $w_m^{sd,\kappa}$ being a Lipschitz continuous and bounded function on $\mathbb{R}^{k\times J_m}$.

The measure $\overline{\mathbb{R}}^{sd}_{S}\big|_{ \mathcal{D}^d \times \mathcal{D}^1 \times \mathcal{D}^k } = \mathbb{Q} \circ \left(X^{u^{sd,\kappa}}, L^{u^{sd,\kappa}}, Y \right)^{-1}$ is clearly the unique probability measure on $\mathcal{D}^d \times \mathcal{D}^1 \times \mathcal{D}^k$ under which: $X_0=x_0$; and $Y$ is a standard Brownian motion; and for every $f \in C^{\infty}_b([0,T] \times \mathbb{R}^d)$, the process
\begin{equation}\label{kappamartingale2}
f(t,X_t) - f(0,X_0) - \int_0^t \cL^{s,X_s,u^{sd,\kappa}_s (Y)} f(s,X_s)\,ds
\end{equation}
is a $\left( {\mathscr{C}}^d_t \otimes {\mathscr{C}}^1_t\right) \vee {\mathscr{C}}^k_T $-martingale with $\cL$ being defined by~\eqref{cLB}; and for every $g\in C_b^{\infty}([0,T]\times \mathbb{R} \times \mathbb{R}^k)$, the process
\begin{equation}\label{kappamartingale3}
g(t,L_t,Y_t)-g(0,L_0,Y_0)-\int_0^t \cM^{s,X_s,L_s}g(s,L_s,Y_s)\,ds
\end{equation}
is a ${\mathscr{C}}^d_t \otimes {\mathscr{C}}^1_t \otimes {\mathscr{C}}^k_t $-martingale with $\cM$ being defined by~\eqref{cML}.

For $h\in\left\{ \frac{\kappa}{2^l} : l\in\mathbb{N}_+\right\}$, we define
$$
\overline{u}_t^h = u^{sd,\kappa}_t \left(\widehat{Y}^h\right) := {w}_m^{sd,\kappa} \left(\widehat{Y}^h_{r_j^m}\,,\, j\leq J_m \,,\, r_j^m \leq m\kappa \right),\quad\quad t\in[m\kappa,(m+1)\kappa),
$$
where $\left\{ \widehat{Y}^h_{r_j^m} \,,\, j\leq J_m \,,\, r_j^m \leq m\kappa \right\}$ is determined by $\left\{ Y^h_i \right\}_{0\leq t_i \leq m \kappa}$. Then $\overline{u}^h\in \cU_h$ with
\begin{equation}\label{pfRhuhQ}
\overline{\mathbb{R}}_{h,\overline{u}^h} = \mathbb{Q}^h \circ \left( \widehat{X}^{h,\overline{u}^h}, \widehat{L}^{h,\overline{u}^h}, {\widehat{Y}^h}, q^{h,\overline{u}^h} \right)^{-1} \in\cR_h.
\end{equation}
According to Lemma~\ref{Rhtight}, we know that $\left\{ \overline{\mathbb{R}}_{h,\overline{u}^h} \right\}_{h\in\left\{ \frac{\kappa}{2^l} : l\in\mathbb{N}_+\right\}}$ is tight. Suppose that its convergent subsequence $\overline{\mathbb{R}}_{h,\overline{u}^h} \longrightarrow \overline{\mathbb{Q}}$. Lemma~\ref{Rhtight} also tells us that $\overline{\mathbb{Q}} \in \cR$.

Then by the same arguments as in proving~\eqref{hmartingale}, together with the uniqueness of the solution to the martingale problem associated to~\eqref{kappamartingale2} and~\eqref{kappamartingale3}, we have 
\begin{equation}\label{pflemmarhuhddk}
\overline{\mathbb{R}}_{h,\overline{u}^h}\big|_{ \mathcal{D}^d \times \mathcal{D}^1 \mathcal{D}^k} \longrightarrow \overline{\mathbb{R}}^{sd}_{S}\big|_{ \mathcal{D}^d \times \mathcal{D}^1 \times \mathcal{D}^k }.
\end{equation}

According to~\eqref{pfRhuhQ}, it holds that for any $\Phi \in C_b(\overline{\Omega})$,
\be\nonumber
\displaystyle \mathbb{E}^{\overline{\mathbb{R}}_{h,\overline{u}^h}} \left[ \Phi \left(X,L,Y,q\right) \right] &=&\displaystyle \mathbb{E}^{\mathbb{Q}^h} \left[ \Phi \left( \widehat{X}^{h,\overline{u}^h}, \widehat{L}^{h,\overline{u}^h}, {\widehat{Y}^h}, q^{h,\overline{u}^h} \right) \right] = \mathbb{E}^{\mathbb{Q}^h} \left[ \Phi \left(\widehat{X}^{h,\overline{u}^h}, \widehat{L}^{h,\overline{u}^h}, {\widehat{Y}^h}, q^{{u}^{sd,\kappa}(\widehat{Y}^h)}\right) \right] \\\nonumber
&=&\displaystyle \mathbb{E}^{\overline{\mathbb{R}}_{h,\overline{u}^h}} \left[ \Phi \left(X,L,Y, q^{{u}^{sd,\kappa}(Y)}\right) \right] = \mathbb{E}^{\overline{\mathbb{R}}_{h,\overline{u}^h}} \left[ \Phi \left(X,L,Y,\delta_{{u}_t^{sd,\kappa}(Y)}(da)dt \right) \right].
\ee
Since the functions $\left\{{w}_m^{sd,\kappa}\right\}$ are Lipschitz continuous, by~\eqref{pflemmarhuhddk} we have
\be\nonumber
\displaystyle \mathbb{E}^{\overline{\mathbb{Q}} } \left[ \Phi(X,L,Y,q) \right] &=&\displaystyle \lim_{h\rightarrow 0} \mathbb{E}^{\overline{\mathbb{R}}_{h,\overline{u}^h}} \left[ \Phi(X,L,Y,q) \right] = \lim_{h\rightarrow 0} \mathbb{E}^{\overline{\mathbb{R}}_{h,\overline{u}^h}} \left[ \Phi \left(X,L,Y,\delta_{{u}_t^{sd,\kappa}(Y)}(da)dt \right) \right]  \\\nonumber
&=&\displaystyle  \mathbb{E}^{\overline{\mathbb{R}}^{sd}_{S}} \left[ \Phi \left(X,L,Y,\delta_{{u}_t^{sd,\kappa}(Y)}(da)dt \right) \right] = \mathbb{E}^{\overline{\mathbb{R}}^{sd}_{S}} \left[ \Phi \left(X,L,Y,q \right) \right].
\ee
Thus $ \overline{\mathbb{Q}} = \overline{\mathbb{R}}^{sd}_{S}$, and we have $\overline{\mathbb{R}}_{h,\overline{u}^h} \longrightarrow \overline{\mathbb{R}}^{sd}_{S}$ as $h\rightarrow 0$. This completes the proof.
\end{proof}

Finally, we are ready to prove our main result as follows.

\noindent \textbf{Proof of Theorem~\ref{ConvergenceTheorem}}. Note that $\Gamma\in \overline{L}^1_*$ by Lemma~\ref{GammainL1*}. In view of Lemma~\ref{L1*convergence} and Lemma~\ref{Rhtight}, we have
\begin{equation}\nonumber
\varlimsup_{h\rightarrow 0} \,\,\sup_{\overline{\mathbb{R}}_{h,{u^h}} \in \cR_h} \mathbb{E}^{\overline{\mathbb{R}}_{h,{u^h}}} \left[ \Gamma \right] \leq \sup_{\overline{\mathbb{R}} \in \cR} \mathbb{E}^{\overline{\mathbb{R}}} \left[ \Gamma \right]=V_R.
\end{equation}

According to Lemma~\ref{constructRs0}, we have
\begin{equation}\nonumber
\varliminf_{h\rightarrow 0} \,\,\sup_{\overline{\mathbb{R}}_{h,{u^h}} \in \cR_h} \mathbb{E}^{\overline{\mathbb{R}}_{h,{u^h}}} \left[ \Gamma \right] \geq \sup_{\overline{\mathbb{R}}^{sd}_{S} \in \cR^{sd}_{S}} \mathbb{E}^{\overline{\mathbb{R}}^{sd}_{S}} \left[ \Gamma \right] =V^{sd}_{S}.
\end{equation}

By~\eqref{defbarVhcanonical} and Proposition~\ref{equalRSsd}, we get
$$
V_R=V^{sd}_{S}\leq \varliminf_{h\rightarrow 0} {V}_h \leq \varlimsup_{h\rightarrow 0} {V}_h \leq V_R.
$$
Thus $V_S = V_R = \lim\limits_{h\rightarrow 0} {V}_h$ according to~\eqref{equalSWR}. This completes the proof of Theorem~\ref{ConvergenceTheorem}.

\section*{Acknowledgments} The research of Xiaolu Tan is supported by Hong Kong RGC General Research Fund (project 14302921).

\bibliographystyle{siam}
\bibliography{references}

\end{document}